\documentclass[11pt]{amsart}
\usepackage{mathrsfs}
\usepackage{mathtools}
\usepackage{dsfont}
\usepackage{amsmath,amssymb,amsthm,upref,graphicx,mathrsfs}
\usepackage{enumerate}

\usepackage{bbm}

\usepackage{color}
\usepackage[
  colorlinks=true,
  linkcolor=blue,
  citecolor=blue,
  urlcolor=blue]{hyperref}

\usepackage{appendix}

\numberwithin{equation}{section}

\newtheorem{theorem}{Theorem}[section]
\newtheorem{proposition}[theorem]{Proposition}
\newtheorem{lemma}[theorem]{Lemma}
\newtheorem{corollary}[theorem]{Corollary}

\theoremstyle{definition}
\newtheorem{definition}[theorem]{Definition}
\newtheorem{example}[theorem]{Example}
\newtheorem{remark}[theorem]{Remark}
\newtheorem{question}[theorem]{Question}
\newtheorem{assumption}[theorem]{Assumption}

\newcommand{\Ex}{\mathcal{E}}
\newcommand{\M}{\mathcal{M}}
\newcommand{\1}{\mathbf{1}}
\newcommand{\Ra}{\mathcal{R}}
\newcommand{\Tr}{\mathcal{T}^{\mathcal{R}}}
\newcommand{\T}{\mathbb{T}}

\newcommand{\R}{\mathbb{R}}

\newcommand{\Na}{\mathbb{N}}

\newcommand{\be}{\begin{eqnarray*}}
\newcommand{\ee}{\end{eqnarray*}}
\newcommand{\beq}{\begin{equation}}
\newcommand{\eeq}{\end{equation}}

\newcommand{\N}{\mathcal{N}}

\begin{document}

\title[Almost uniform convergence for  noncommutative Vilenkin-Fourier series]
{Almost uniform convergence for  noncommutative Vilenkin-Fourier series}

\author[Jiao]{Yong Jiao}
\address{School of Mathematics and Statistics,  HNP-LAMA, Central South University, Changsha 410083, China}
\email{jiaoyong@csu.edu.cn}

\author[Luo]{Sijie Luo}
\address{School of Mathematics and Statistics, HNP-LAMA, Central South University, Changsha 410083, China}
\email{sijieluo@csu.edu.cn}

\author[Zhao]{Tiantian Zhao}
\address{Institute for Advanced Study in Mathematics, Harbin Institute of Technology, Harbin 150001;
School of Mathematics and Statistics, Central South University, Changsha 410083, China}
\email{zhaotiantian@hit.edu.cn}

\author[Zhou]{Dejian Zhou}
\address{School of Mathematics and Statistics, HNP-LAMA,
Central South University, Changsha 410083, China}
\email{zhoudejian@csu.edu.cn}
%




%

\begin{abstract}
In the present paper, we study almost uniform convergence for  noncommutative Vilenkin-Fourier series. Precisely, we establish several  noncommutative (asymmetric) maximal inequalities for the Ces\`{a}ro means of the noncommutative Vilenkin-Fourier series, which in turn give the corresponding almost uniform convergence.
 The primary strategy in our proof  is to explore  a noncommutative  generalization of  Sunouchi square function operator, and the very recent advance of the noncommutative Calder\'{o}n-Zygmund decomposition.
\end{abstract}

\maketitle

%
%

 \section{Introduction}
The problem of pointwise convergence for the Fourier series is one of the  
central and oldest themes in harmonic analysis. For $f\in L_1(\T)$, the partial sum of the Fourier series of $f$ is defined by
\[
S_n(f)(t)\coloneqq\sum_{k=-n}^n\widehat{f}(k)e^{2\pi i k t}, \quad n\in \mathbb{N},\,t\in \mathbb{T},
\]
where $\widehat{f}(k)=\int_\T f(y)e^{-2\pi iky}dy$ is the $k$-th Fourier coefficient of $f$.
Carleson and Hunt  (\cite{Ca1966, Hu1968}) showed that  $S_n(f)$ converges a.e. to $f$ (as $n\to \infty$) provided  $f\in L_p(\T)$ for $1<p<\infty$. However, in 1923, it was already shown by Kolmogorov that there exists $g\in L_1(\T)$ such that $S_n(g)$ diverges almost everywhere; see e.g. \cite[Theorem 3.4.2]{Ga2008}. Hence, it is more suitable to consider some other summability methods for  functions in $L_1(\mathbb{T})$. For instance, the so-called Ces\`{a}ro means (or Fej\'{e}r means) of Fourier series are defined as follows: for $f\in L_1(\mathbb{T})$ and $n\in \mathbb{N}$,
\[
\sigma_{n}(f)\coloneqq\frac{1}{n}\sum_{j=0}^{n-1}S_{j}(f).
\]
In 1905, Lebesgue \cite{Le1905} extended a Fej\'{e}r's result (see e.g. \cite[Theorem 3.3.1]{Ga2008}) by obtaining that, for each $f\in L_{1}(\T)$,
\begin{equation}\label{FL04-05}
\sigma_{n}(f)\to f,\quad \mbox{a.e.}, \quad (\mbox{as}~n\to\infty).
\end{equation}

\smallskip
Recently, due to broad interactions with quantum information theory, group theory, and noncommutative geometry, the noncommutative Fourier analysis has become an active area, and numerous results concerning pointwise convergences of the Fourier series have been established successfully in the noncommutative setting. In 2013, Chen, Xu and Yin  \cite[Theorem 4.2]{CXY2013} obtained the noncommutative version of \eqref{FL04-05} on   quantum tori, that is, the Ces\`{a}ro means  converge bilaterally almost uniformly (b.a.u.) with initial data in $ L_1(\mathbb{T}_{\theta}^d)$, where $\mathbb{T}_{\theta}^d$ is a quantum torus. We refer   readers to Definition \ref{npc} in the appendix for different types of almost uniform convergences in the noncommutative setting. Recently, Hong et al.  \cite{HWW2019} studied the pointwise convergence of Fourier series for group von Neumann algebras and quantum groups. Specifically, they developed a general criterion of maximal inequalities for approximate identities of noncommutative multipliers, which allows them to show the b.a.u. (or a.u.) convergence for associated Fourier multipliers on noncommutaitve $L_{p}$-spaces ($1<p<\infty$), under certain conditions. In the study of multiparametric b.a.u. convergence,  Conde-Alonso et al. \cite{CGP2020} established   noncommutative correspondings of the Jessen-Marcinkiewicz-Zygmund theorem and the C\'{o}rdoba-Fefferman-Guzm\'{a}n inequality for quite general processes. As a consequence, they obtained the b.a.u. convergence for free Poisson semigroup with initial element in $L\log^2 L(\M) $. 

In the present paper, we firstly focus on the Ces\`{a}ro summability for  the  Vilenkin-Fourier series of operator-valued functions, and then transfer the operator-valued results into their corresponding totally noncommutative setting. 
The operator-valued (or semi-commutative) model often provides deep insights in noncommutative harmonic analysis, and sometimes plays essential role based on the transference principles. For example, in the paper \cite{CXY2013}  (see also \cite{CGP2020}), the authors firstly proved noncommutative maximal inequality for Ces\`{a}ro means  in operator-valued setting, and then transferred the results to quantum tori. 
 We also refer the reader to the papers \cite{CCP2022, CGPT2023, HLW2021, La2022, XXX2016} whose main ideas are to reduce the problems in their setting to the corresponding problems in the operator-valued setting.

 In order to explain the classical theory roots of those we investigate here, we need to recall several remarkable results.  Historically, the classical  Walsh system  was first introduced by Paley  \cite{Pa1932}, and the classical Vilenkin system was first introduced by Vilenkin \cite{Vi1947}. In the  classical setting,  Walsh system  and Vilenkin system has been studied extensively in the harmonic analysis; see for instance \cite{Golubov1991,PTW2022,SWS1990}.  Motivated by  Carleson and Hunt's results  \cite{Ca1966, Hu1968}, it has been shown that the partial sums of the  Walsh-Fourier series  of an $L_p(0,1)$-function converge almost everywhere to the function for $1<p<\infty$ (see \cite{Bi1966, Hu1971, Sj1969}), and such convergence result fails for the case $p=1$. However, similar to the Fourier case, if we consider some other summability methods, we could still obtain the almost convergence for $L_{1}(0,1)$ functions. Let $(\sigma_{n}(f))_{n\geq1}$ be the Ces\`{a}ro means for the Walsh-Fourier series of $f\in L_1(0,1)$, more precisely,
\[
\sigma_n(f)(\eta)=\int_0^1 f(t) K_n(\eta \oplus t)  dt,~\eta\in(0,1),
\]
where  $K_n$ is the Walsh-Fej\'er kernel and $\oplus$ is the dyadic addition. In the substantial paper \cite{Fi1955},  Fine proved that $\sigma_{n}(f)$ converges almost everywhere to $f$ for each $f\in L_1(0,1)$. Besides, Schipp \cite{Sc1975} investigated Fine's result with an entirely different approach and deduced the convergence theorem via the following weak type $(1,1)$  maximal inequality:
\begin{equation}\label{Sch}
\Big\|\sup_{n\geq1} |\sigma_n(f)|\Big\|_{L_{1,\infty}(0,1)}\leq c\|f\|_{L_1(0,1)}, \quad f\in L_1(0,1).
\end{equation}
By interpolation, \eqref{Sch} further implies the following strong type $(p,p)$ inequality which was  obtained by  Sunouchi \cite{Su1951} and Yano \cite{Ya1957} independently:
\begin{equation}\label{Ya}
\Big\|\sup_{n\geq1} |\sigma_n(f)|\Big\|_{L_{p}(0,1)}\leq c_p\|f\|_{L_p(0,1)}, \quad f\in L_p(0,1), \quad 1<p\leq \infty.
\end{equation}
Furthermore, in 1979, Fujii \cite{Fu1979}  proved that 
\begin{equation}\label{Fujii}
\Big\|\sup_{n\geq1} |\sigma_n(f)|\Big\|_{L_{1}(0,1)}\leq c\|f\|_{H_1(0,1)},
\end{equation}
where $H_1(0,1)$ denotes the martingale Hardy space. 
The corresponding results mentioned above have also been established in the Vilenkin system; we refer the reader to \cite{Go1973} for the almost convergence of partial sums in bounded Vilenkin system (see also \cite{PSTW2022} for a new proof);   to \cite{PTW2022} for more related results. Furthermore, G\'{a}t \cite{Ga2001} established \eqref{Sch} and \eqref{Fujii} on Vilenkin-like systems.


Motivated by \eqref{Sch}-\eqref{Fujii}, our first object is to deal with noncommutative maximal inequalities associated with the specific family of Ces\`{a}ro means of operator-valued Fourier series in the so-called Vilenkin-like system (a common generalization of the Walsh system, Vilenkin system,
the character system of the group of 2-adic (m-adic) integers, the product system of normalized
coordinate functions for continuous irreducible unitary representations of the coordinate
groups of noncommutative Vilenkin groups, the UDMD product systems; see e.g. \cite{Ga2001} for more information).   Note that even the definition of the maximal function $\sup_nf_n$ is not available in the noncommutative setting. The problems of noncommutative maximal inequalities are often  non-trivial, and the usual techniques in classical analysis involving maximal functions seem no
longer available in the noncommutative case.  This  applies to noncommutative martingale Doob inequality \cite{Ju2002}, noncommutative  maximal ergodic theorems \cite{JX2007}, noncommutative Hardy-Littlewood maximal inequalities \cite{Me2007}, and also to our topic in the present paper (see below for more information).

Our first main result of this paper  is the following noncommutative version of \eqref{Sch} and \eqref{Ya} for Ces\`{a}ro means of   Vilenkin-like-Fourier series of operator-valued functions. In the sequel, let $\mathcal{N}=L_{\infty}(G_m)\bar{\otimes} \mathcal{M}$, where $G_m$ is a  Vilenkin space associated with a  bounded sequence $m=(m_k)_{k\in\mathbb{N}}\subset\mathbb{N}$ (see \eqref{eq:vs} below for details) and $\mathcal{M}$ is a semifinite von Neumann algebra. The definitions of $\|\cdot\|_{\Lambda_{1,\infty}(\mathcal{N},\ell_{\infty})}$,  $\|\cdot\|_{L_p(\mathcal{N}, \ell_{\infty})}$ and any unspecific notations are referred to Section \ref{sec-2}.

\begin{theorem}\label{CS}
	Given a bounded sequence $m=(m_k)_{k\in\mathbb{N}}\subset\mathbb{N}$,	let $\mathcal{N}=L_{\infty}(G_m)\bar{\otimes} \mathcal{M}$. 	Let $(\sigma_{n})_{n\geq1}$ be the Ces\`{a}ro means of the  Vilenkin-like-Fourier series of functions in $L_p(\mathcal{N})$.
	\begin{enumerate}[{\rm (i)}]
		\item If $p=1$, then there is a constant $c>0$ such that 	$$\|(\sigma_{n}(f))_{n\geq1}\|_{\Lambda_{1,\infty}(\mathcal{N},\ell_{\infty})}\leq c \|f\|_{L_1(\mathcal{N})}.$$
		\item If $1<p\leq \infty$, then  there is a constant $c_p>0$ such that  $$\|(\sigma_{n}(f))_{n\geq1}\|_{L_p(\mathcal{N},\ell_{\infty})}\leq  c_p\|f\|_{L_p(\mathcal{N})},$$
		where $c_p=\frac{cp^2}{(p-1)^2}$ for some positive constant $c$.
	\end{enumerate}
\end{theorem}

There appear several difficulties to prove this theorem. The pointwise estimate used in the commutative proof \cite{Ga2001} does not work for our purpose now.  As mentioned above, Vilenkin-like system is a common generalization of the Walsh system, Vilenkin system,
the character system of the group of 2-adic (m-adic) integers and other systems. This leads that the Fej\'er kernels in Vilenkin-like system do not behave as well as the Fej\'er kernels on the torus $\mathbb{T}$ which are dominated by decreasing functions whose integrals are bounded. In fact, for the time
being there is no formula for the Fej\'er kernels in the Vilenkin-like system.  Hence we can not  reduce our problem to Mei's noncommutative  Hardy-Littlewood maximal inequalities as done in \cite{CXY2013}.  Finally, we point out that  the  Ces\`{a}ro means $(\sigma_{n})_{n}$ do not even  preserve self-adjointness. 
To overcome or avoid these difficulties, we reduce Theorem \ref{CS} to Theorem \ref{sigma-tilde}; see the argument before Theorem \ref{sigma-tilde}.
 One of the key  ingredients in our proof  is to exploit the new noncommutative Calder\'{o}n-Zygmund decomposition  (see Theorem \ref{NewCZ} below)  established very recently by Cadilhac et al.  \cite{CCP2022}. The great advantage of the improved decomposition is that one can circumvent the pseudo-localization technique, which may  not work well for noncommutative maximal inequalities. Recently, noncommutative weak type maximal inequalities for singular integrals were established in \cite{HLX2023}. However, the  Fej\'er kernels for Vilenkin-like system does not fulfill the $L_q$-integral regularity condition used in \cite{HLX2023}, where such condition plays a significant role.  In addition, in order to apply the noncommutative Calder\'{o}n-Zygmund decomposition, we also need to further explore new properties for the  Fej\'er kernels of the Vilenkin-like system; see Lemmas \ref{K-E-t} -- \ref{KI-E-lemma}.

\smallskip
Our second  main objective is the noncommutative analogy of the Fujii inequality \eqref{Fujii}. The symbol $H_p^c(\M)$  denotes the noncommutative martingale Hardy space; see Section \ref{sec-pf-main1}.  It seems that the following inequality does not hold in general
\begin{equation*}
	\left\|(\sigma_k(f))_{k\geq1}\right\|_{L_1(\mathcal{N},\ell_{\infty})}\leq c \|f\|_{H_1^c(\N)},
\end{equation*}
since such kind of inequality fails even for noncommutative martingales (\cite[Lemma 13]{JX2005}). 
Inspired by \cite{HJP2016}, it is natural  to consider asymmetric maximal inequalities for the Ces\`{a}ro means $(\sigma_k)_{k\geq1}$ as follows: for $1\leq p\leq 2$,
\begin{equation}\label{full-range}
	\|(\sigma_{k}(f))_{k\geq1}\|_{\Lambda_{p,\infty}(\mathcal{N},\ell_{\infty}^c)}\leq c_{p}\|f\|_{H_p^c(\mathcal{N})}.
\end{equation}
A crucial difficulty to show this inequality is that we could not take the advantage of the self-adjointness-preserving nature of the  Ces\`{a}ro means $(\sigma_k)_{k\geq1}$ directly. Besides, the powerful algebraic atomic decomposition technique used   in \cite{HJP2016} is not available in our setting. Indeed, by a well-known fact in \cite{Ju2002} that each conditional expectation $\Ex:\M\to \N$ can be rewritten as: for some linear isometry $u:L_{2}(\M)\to L_{2}(\mathcal{N}\bar{\otimes}B(\ell_{2}))$, 
\[
\mathcal{E}(x^*y)=u(x)^*u(y),\quad x,y\in L_2(\M),
\]
which plays an important role in establishing asymmetric inequalities for noncommutative martingales. However, the failure of such property for the Ces\'aro means $(\sigma_{k})_{k\geq 1}$ obstructs us from using the algebraic atomic decomposition technique.
  Due to such difficulties, we can not prove \eqref{full-range} at the time of this writing. Fortunately, via appropriate estimate of a certain noncommutative square function, we can obtain  a lacunary version of \eqref{full-range}.  Recall that a sequence $(n_k)_{k\geq1}\subseteq \mathbb{N}$ is called lacunary if there exists a number $r>1$ such that $n_{k+1}/n_{k}>r$ for all $k\geq1$.

\begin{theorem}\label{main-asy-op}
	Given a bounded sequence $m=(m_k)_{k\in\mathbb{N}}\subset\mathbb{N}$,	let $\mathcal{N}=L_{\infty}(G_m)\bar{\otimes} \mathcal{M}$. 	Let $(\sigma_{n})_{n\geq1}$ be the Ces\`{a}ro means of the  Vilenkin-like-Fourier series of functions in $L_p(\mathcal{N})$, and let $(n_{k})_{k\geq 1}$ be a lacunary sequence.
	\begin{enumerate}[{\rm (i)}]
		\item For $1\leq p\leq  2$, there exists $c_{p}>0$ such that 
		\[
		\|(\sigma_{n_{k}}(f))_{k\geq1}\|_{\Lambda_{p,\infty}(\mathcal{N},\ell_{\infty}^c)}\leq c_{p}\|f\|_{H_p^c(\mathcal{N})}.
		\]
		The same holds for row spaces. 
		
		\item For $1< p<2$, there exists $c_{p}>0$ such that
		\[
		\inf_{f=f^{c}+f^{r}}\{\|(\sigma_{n_{k}}(f^{c}))_{k\geq1}\|_{L_{p}(\mathcal{N},\ell_{\infty}^{c})}+\|(\sigma_{n_{k}}(f^{r}))_{k\geq1}\|_{L_{p}(\mathcal{N},\ell_{\infty}^{r})}\}\leq c_{p}\|f\|_{L_p(\mathcal{N})}.
		\]
	\end{enumerate}
\end{theorem}

Our proof of the  asymmetric inequalities in Theorem \ref{main-asy-op} here entirely differs from the  approach in \cite{HJP2016} but relies on estimates of  a concrete noncommutative square function -- a noncommutative generalization of Sunouchi operator.  Our primary strategy is inspired by \cite[Chapter XV]{Zy}, where maximal inequality was established via Littlewood-Paley theory. We explain the key idea here in noncommutative language.  By the basic result Lemma \ref{ebm-pc}, we have
\begin{align*}
\|(\sigma_{n_{k}}(f))_{k\geq1}\|_{\Lambda_{p,\infty}(\mathcal{N},\ell_{\infty}^c)} &\lesssim  \|[\sigma_{n_{k}}(f)-\mathbb{E}_k(f)]_{k\geq1}\|_{\Lambda_{p,\infty}(\mathcal{N},\ell_{\infty}^c)}+\|(\mathbb{E}_k(f))_{k\geq1}\|_{\Lambda_{p,\infty}(\mathcal{N},\ell_{\infty}^c)} \\
&\leq \|[\sigma_{n_{k}}(f)-\mathbb{E}_k(f)]_{k\geq1}\|_{L_p(\mathcal{N},\ell_2^c)}+ \|(\mathbb{E}_k(f))_{k\geq1}\|_{\Lambda_{p,\infty}(\mathcal{N},\ell_{\infty}^c)},
\end{align*}
where $(\mathbb{E}_{k})_{k\geq 1}$ denote the  conditional expectations are given as in \eqref{ce}. Using the available result for noncommutative martingales \cite[Theorem A]{HJP2016},  it suffices   to prove the following inequality
\begin{equation}\label{square}
\|U(f)\|_{L_p(\N)}=\|[\sigma_{n_{k}}(f)-\mathbb{E}_k(f)]_{k\geq1}\|_{L_p(\mathcal{N},\ell_2^c)}\leq  c_p\|x\|_{H_p^c(\mathcal{N})},
\end{equation}
where
	$$U(f)= \Big(\sum_{k\geq1}|\sigma_{n_k}(f)-\mathbb{E}_k(f)|^2\Big)^{1/2}.$$
This operator  was firstly 
studied by  Sunouchi in \cite{Su1951} for the classical Walsh setting, i.e. $n_k=2^k$, where the equivalence $\|U(f)\|_{L_p(G_m)} \sim_p \|f\|_{L_p(G_m)}$, $1<p<\infty$, was established. This  so-called Sunouchi operator actually plays an important role for strong summability of Fourier series; see e.g. \cite{Su1964}. In the classical  bounded Vilenkin system,  the Sunouchi operator  was investigated by G\'{a}t in \cite{Ga1993}, where he established the $H_1$-$L_1$ boundedness of $U$. We also refer the reader to \cite{Si2000} and \cite{We2005} and their references about this topic. 
We need to point out that even for $\mathcal{M}=\mathbb{C}$, this is the first time to study such generalization of the Sunouchi operator $U$ in the Vilenkin-like system.  Note that now  the  Fej\'er kernels of the Vilenkin-like system do not have good properties as the ones in Vilenkin case; also noncommutativety causes additional difficulties. The usual argument in available literature such as \cite{Su1951, Ga1993, Si2000, We2005} does not work well to obtain \eqref{square}. We finally prove \eqref{square} in Theorem \ref{SO-1} with the help of new properties of the  Fej\'er kernels of the Vilenkin-like system (see Lemma \ref{K^2-E-a}). Of course, we apply the powerful tool, i.e., noncommutative atomic decomposition of martingale Hardy spaces. 
 We also point out that Theorem \ref{SO-1} is  even new in the classical case.

Now we turn to the main theme of this paper, that is, the corresponding results of Theorem \ref{CS} and Theorem \ref{main-asy-op} in totally noncommutative setting.
Before going further, we first recall several related results.
In the series of papers by  Sukochev and his collaborators \cite{SF1995,AFS1996,CPS2013},  appropriate noncommutative generalizations of the  Walsh system for hyperfinite $\mathrm{II}_1$, $\mathrm{II}_{\infty}$ and $\mathrm{III}_{\lambda}$ $(0<\lambda\leq 1)$ factors  have been introduced. Moreover, they proved that  the corresponding noncommutative  Walsh system is a Schauder basis in the $L_p$-spaces for $1<p<\infty$. In \cite{JZWZ2018}, Jiao et al. showed that the sequence of partial sums of the corresponding noncommutative Walsh-Fourier series in the hyperfinite $\mathrm{II}_1$ factor  is a uniformly bounded family of weak type $(1,1)$ operators. Moreover, Wu \cite{Wu2016} studied multipliers for noncommutative Walsh-Fourier series. Considering the results related to the noncommutative Walsh system in the CAR algebra, we refer interested readers to \cite{CL1993} for the hypercontractivity, to \cite{Lu1998} for the Riesz transform, and to \cite{BL2008} for Poincar\'{e} type inequalities. In addition,  Sukochev and his collaborators \cite{DS2000,DFdePS2001,SS2018} also studied  noncommutative generalizations of the Vilenkin  system for the hyperfinite $\mathrm{II}_1$ factor, and they proved that partial sums operator of  the corresponding noncommutative Vilenkin-Fourier series is  convergent in the $L_p$-spaces for $1<p<\infty$. Recently, the authors in \cite{TZ2023} proved the main result of \cite{SS2018} via transference method. However, to the best of our knowledge, there is no  any available literature considering pointwise convergence (almost uniform convergence) of the noncommutative Vilenkin-Fourier series; one of our motivations is to fill such gaps.

We transfer Theorem \ref{CS} and Theorem \ref{main-asy-op} into noncommutative Vilenkin system setting. In what follows, $\mathcal{R}$ denotes the  the hyperfinite $\mathrm{II}_1$ factor.
\begin{theorem}\label{Nc-W}
	Let $(\sigma^{\Ra}_{n})_{n\geq1}$ be the Ces\`{a}ro means of the noncommutative Vilenkin-Fourier series. 
	\begin{enumerate}[{\rm (i)}]
		\item There exists a universal constant $c>0$ such that
		\[
		\left\|(\sigma^{\mathcal{R}}_n(x))_{n\geq1}\right\|_{\Lambda_{1,\infty}(\mathcal{R},\ell_{\infty})}\leq c \|x\|_{L_1(\mathcal{R})}.
		\]
		\item For $1<p\leq \infty$, 
		\[
		\left\|(\sigma^{\mathcal{R}}_n(x))_{n\geq1}\right\|_{L_{p}(\mathcal{R},\ell_{\infty})}\leq \frac{cp^2}{(p-1)^2} \|x\|_{L_p(\mathcal{R})}.
		\]
	\end{enumerate}
\end{theorem}

With the help of Theorem \ref{Nc-W}, the following result, whose proof is omitted here, can be proved similarly to that in \cite[Theorem 5.1]{CXY2013}.
\begin{corollary}\label{Nc-CSbau}
	For $1\leq p<\infty$ and $x\in L_{p}(\Ra)$, $\sigma_{n}^{\Ra}(x)$ converges to $x$ b.a.u. as $n\to\infty$. Moreover, for $2\leq p< \infty$, the b.a.u. convergence can be strengthened to the a.u. convergence.
\end{corollary}

\begin{theorem}\label{main-asy}
	Let $(\sigma^{\Ra}_{n})_{n\geq1}$ be the Ces\`{a}ro means of the noncommutative Vilenkin-Fourier series and $(n_{k})_{k\geq 1}$ be a lacunary sequence.
	\begin{enumerate}[{\rm (i)}]
		\item For $1\leq p\leq  2$, there exists $c_{p}>0$ such that 
		\[
		\|(\sigma^{\Ra}_{n_{k}}(x))_{k\geq1}\|_{\Lambda_{p,\infty}(\mathcal{R},\ell_{\infty}^c)}\leq c_{p}\|x\|_{H_p^c(\mathcal{R})}.
		\]
		The same holds for row spaces. 
		
		\item For $1< p<2$, there exists $c_{p}>0$ such that
		\[
		\inf_{x=x^{c}+x^{r}}\{\|(\sigma^{\Ra}_{n_{k}}(x^{c}))_{k\geq1}\|_{L_{p}(\mathcal{R},\ell_{\infty}^{c})}+\|(\sigma^{\Ra}_{n_{k}}(x^{r}))_{k\geq1}\|_{L_{p}(\mathcal{R},\ell_{\infty}^{r})}\}\leq c_{p}\|x\|_{L_p(\mathcal{R})}.
		\]
	\end{enumerate}
\end{theorem}

As a consequence of Theorem \ref{main-asy}, we obtain the following convergence result, whose proof is  analogous to \cite[Proposition 6]{De2011}. Hence, we omit the proof.  

\begin{corollary}\label{cr-au}
	Let   $(n_{k})_{k\geq 1}$ be a lacunary sequence.
	\begin{enumerate}[{\rm (i)}]
		\item For any $x\in H_1^c(\Ra)$,
		$\sigma^{\Ra}_{n_k}(x)$ converges to $x$  c.a.u.  as $k\to\infty$. Similarly, when $x\in H_1^r(\Ra)$, r.a.u. convergence holds. 
		\item For any $x\in L_p(\Ra)$ with $1<p<2$,
		$\sigma^{\Ra}_{n_k}(x)$ converges to $x$  column $+$ row  a.u.  as $k\to\infty$. 
	\end{enumerate}
\end{corollary}

 Here we mention the main difference between our results and \cite{HWW2019}. In \cite{HWW2019} (see Theorems 3.2 and 3.3 there), Hong et al. only consider noncommutative almost uniform convergence for the elements in $L_p$ with the range $1<p<\infty$. However, our results, i.e., Theorem \ref{Nc-W} and Corollary \ref{Nc-CSbau}, work for elements in $L_1(\mathcal{R})$. As for $1<p<2$, \cite[Theorem 3.3]{HWW2019} shows that, under certain criterion, lacunary Fourier series in $L_p$ (for group von Neumann algebras or quantum groups) converge bilaterally almost uniformly. Our result Corollary \ref{cr-au} (ii) establishes stronger column $+$ row almost uniform convergence for lacunary noncommutative Vilenkin-Fourier series.

Our paper is organized as follows. The preliminaries section (Section \ref{sec-2}) contains necessary definitions and notations, which will be frequently used throughout this paper. Section \ref{noncommutative Walsh and transference} is a   discussion to the transference argument that  enables us to transfer the operator-valued results to corresponding noncommutative ones. The last two sections are devoted to the proofs of main results of this paper. Precisely, we provide the  proofs of Theorem \ref{CS} and Theorem \ref{Nc-W} in Section  \ref{sec-pf-NcW},  the proofs of Theorem \ref{main-asy-op} and Theorem \ref{main-asy} in  Section \ref{sec-pf-main1}, respectively. The almost uniform convergence theorems follow from corresponding noncommutative maximal inequalities via standard arguments (see e.g. \cite{CXY2013, De2011, HLW2021,JX2007}); hence, in the present paper, we omit the proofs of Corollary \ref{Nc-CSbau} and Corollary \ref{cr-au}. We conclude this paper by two questions.



\smallskip
 \textbf{Notations:} Throughout the paper, we use $\mathbb{N}$, $\mathbb{R}$ and $\mathbb{C}$ to stand for the set of natural, real and complex numbers, respectively. The letter $\mathcal M$ will always denote a tracial von Neumann algebra equipped with a normal finite faithful trace $\tau$. The symbol $\mathbf{1}_{\M}$ (simply $\mathbf{1}$) denotes the unit of $\M$.  The set of all projections  in $\mathcal{M}$ is denoted by $\mathcal{P}(\M)$. We always let $\mathcal{N}$ denote the tensor  von Neumann algebra $L_{\infty}(G_m)\bar{\otimes} \mathcal{M}$ equipped with the trace $\varphi=\int\otimes \tau$. For a given sequence $m=(m_k)_{k\in \mathbb{N}}\subset \mathbb{N}$, the numbers $(M_n)_n$ are defined by setting:  $M_0:=1$, and $M_{n+1}=m_nM_n$, $n\in\mathbb{N}$. The symbol $D(\mathcal{F}_n)$ denotes the collection of all  intervals in $\mathcal{F}_n$ with length $M_n^{-1}$; see Example \ref{one-d}. The symbol  $c$ stands for a positive constant which may vary from line to line, and we write $c_{p}$ to emphasize the constant $c_{p}$ depends only on the parameter $p$.

\section{Preliminaries}\label{sec-2}

\subsection{Noncommutative  (weak) Lebesgue spaces} 
 Let $\mathcal{M}$ be a   von Nuemann algebra with a normal faithful
semifinite trace $\tau$.  Let $L_0(\mathcal{M})$ denote the algebra of all $\tau$-measurable operators.    For $1\leq p\leq \infty$, let $L_p(\M,\tau)$ (simply $L_p(\M)$) be the associated noncommutative Lebesgue space. As usual, $L_{\infty}(\mathcal{M})$ is just $\mathcal{M}$ with the usual operator norm. Also recall that 
$$L_p(\M)=\{x\in L_0(\M): \|x\|_{L_p(\M)}<\infty\}$$
with 
$$\|x\|_{L_p(\M)}= [\tau(|x|^p)]^{1/p},$$
where $|x|=(x^*x)^{1/2}$ is the modulus of $x$.

  Suppose that $a$ is a self-adjoint $\tau$-measurable operator and let $a=\int_{-\infty}^{\infty} \lambda d e_{\lambda}$ stand for its spectral decomposition. For any Borel subset $B$ of $\mathbb{R}$, the spectral projection of $a$ corresponding to the set $B$ is defined by $\chi_B(a)=\int_{-\infty}^{\infty} \chi_B(\lambda) d e_{\lambda}$. 
For $x\in L_0({\M})$, the generalized singular value function $\mu(t,x)$ is defined by
$$\mu(t,x)=\inf\big\{s>0:\tau\big(\chi_{(s,\infty)}(|x|)\big)\leq t\big\},\quad t>0.$$
The function $t\mapsto \mu(t,x)$ is decreasing and right-continuous; for more detailed study of the singular value function we refer the reader to \cite{Fa1986}. Of special interest in this paper is  the noncommutative weak Lebesgue space $L_{p,\infty}(\mathcal{M})$, with the quasi-norm
$$\|x\|_{L_{p,\infty}(\M)}=\sup_{t>0}t^{1/p}\mu(t,x)=\sup_{\lambda>0}\lambda[ \tau(\chi_{(\lambda,\infty)}(|x|))]^{1/p}. $$

\subsection{Noncommutative vector-valued spaces}\label{subsec-2-2} In this subsection, we recall the definitions of the  spaces $L_p(\mathcal M,\ell_{\infty}^{\theta})$ and $\Lambda_{p,\infty}(\mathcal{M}, \ell_{\infty}^{c})$ introduced in \cite{Ju2002, PIS7}   (see also \cite{HJP2016}).

For $0\leq \theta\leq 1$ and $1\leq p\leq \infty$, define $L_p(\mathcal{M},\ell_{\infty}^{\theta})$ as  the space of all sequences $x=(x_n)_{n\in \mathbb{N}}$ in $L_p(\mathcal{M})$ for which there exist $a\in L_{p/(1-\theta)}(\mathcal{M})$, $b\in  L_{p/\theta}(\mathcal{M})$ and a bounded sequence $y=(y_{n})_{n\in \mathbb{N}}\subset \M$ such that
$$x_n=ay_n b,\qquad \forall n\in \mathbb{N}.$$
For $x=(x_n)_{n\in \mathbb{N}}\in L_p(\mathcal{M},\ell_{\infty}^{\theta})$,  define
\begin{equation}\label{e-infty}
\|x\|_{L_p(\mathcal M,\ell_{\infty}^{\theta})}=\inf\{\|a\|_{L_{\frac{p}{1-\theta}}(\mathcal M)}\sup_{n\in \mathbb{N}}\|y_n\|_{L_\infty(\M)}\|b\|_{L_{\frac{p}{\theta}}(\mathcal{M})}\},
\end{equation}
where the infimum is taken over all possible factorizations of $x$ as above. In the sequel, $L_p(\mathcal{M},\ell_{\infty}^{1/2})$, $L_p(\mathcal{M},\ell_{\infty}^{1})$ and $L_p(\mathcal{M},\ell_{\infty}^{0})$ will be denoted by $L_p(\mathcal{M},\ell_{\infty})$, $L_p(\mathcal{M},\ell_{\infty}^{c})$ and  $L_p(\mathcal{M},\ell_{\infty}^{r})$, respectively.

Following \cite{HJP2016}, for $1\leq p<\infty$, define $\Lambda_{p,\infty}(\mathcal M,\ell_\infty)$ as the space of all sequences $x=(x_n)_{n\in \mathbb{N}}$ in $L_{p,\infty}(\mathcal M)$  with quasi-norm given by
\begin{equation}\label{Lambda}
\|x\|_{\Lambda_{p,\infty}(\mathcal M,\ell_\infty)}=\sup_{t>0}\inf_{e\in \mathcal {P}(\mathcal M)}\{t\tau(\textbf{1}-e)^{\frac{1}p}: \|ex_{n}e\|_{L_\infty(\mathcal{M})}\leq t,\forall n\in \mathbb{N}\}<\infty.
\end{equation}
For a sequence $x=(x_n)_{n\in \mathbb{N}}$ in $L_0(\M)$, set 
$$\|x\|_{\Lambda_{p,\infty}(\M,\ell_{\infty}^c)}=\sup_{t>0}\inf_{e\in \mathcal {P}(\mathcal M)} \{t\tau(\1-e)^{\frac{1}{p}}: \|x_{n}e\|_{L_{\infty}(\M)}\leq t,\forall n\in \mathbb{N} \}.$$
Define $\|x\|_{\Lambda_{p,\infty}(\M,\ell_{\infty}^r)}=\|x^*\|_{\Lambda_{p,\infty}(\M,\ell_{\infty}^c)}$.

According to the definitions, it is obvious that 
\begin{equation*}
\sup_{n}\|x_{n}\|_{L_{p,\infty}(\mathcal{M})}\leq \|x\|_{\Lambda_{p,\infty}(\mathcal{M},\ell_{\infty})}\leq \|x\|_{L_p(\mathcal M,\ell_{\infty})},\quad \forall 1\leq p<\infty.
\end{equation*}
Besides, if $x=(x_n)_{n\in \mathbb{N}}\in L_{p}(\mathcal M,\ell_{\infty})$ and $x_n=x_n^*$ for each $n$, then (see  \cite[p.\,518]{Di20151} or \cite{JX2007})
\begin{equation}\label{positive}
	\|x\|_{L_p(\mathcal{M},\ell_{\infty})}= \inf\{\|a\|_{L_p(\mathcal M)}: a\in L_p(\mathcal{M}), -a\leq x_n\leq a,\;\forall n\in \mathbb{N} \}.
\end{equation}
\begin{remark}\label{control}
Consider self-adjoint operators $x=(x_n)_{n\in \mathbb{N}}$ and   $y=(y_n)_{n\in \mathbb{N}}$. Assume that $-y_n\leq x_n\leq y_n$ for each $n\in \mathbb{N}$. Note that for self-adjoint operators $a,b\in L_{\infty}(\mathcal{M})$, $-b\leq a\leq b$ implies $\|a\|_{L_{\infty}(\mathcal{M})}\leq \|b\|_{L_{\infty}(\mathcal{M})}$ (see e.g. \cite[Proposition 4.2.8]{KR}).  Then, according to \eqref{Lambda}, it is easy to see that 
$$\|x\|_{\Lambda_{p,\infty}(\mathcal{M},\ell_{\infty})}\leq \|y\|_{\Lambda_{p,\infty}(\mathcal{M},\ell_{\infty})},\quad 1\leq p<\infty.$$
Similarly, by \eqref{positive}, we also have 
$$\|x\|_{L_p(\mathcal M,\ell_{\infty})}\leq \|y\|_{L_p(\mathcal M,\ell_{\infty})},\quad 1\leq p\leq \infty.$$
\end{remark}

\subsection{Noncommutative Hilbert space valued spaces}
Given $x=(x_{k})_{k=1}^{n}$ a sequence of elements in $L_{p}(\mathcal{M})$, for each $1\leq p<\infty$, define $\|x\|_{L_{p}(\mathcal{M};\ell^{c}_{2})}$ as
\begin{equation*}
\|x\|_{L_{p}(\mathcal{M};\ell^{c}_{2})}=\Big\|\Big(\sum\limits_{k=1}^{n}|x_{k}|^{2}\Big)^{1/2}\Big\|_{L_{p}(\mathcal{M})}.
\end{equation*}
Analogously, we could also define
\begin{equation*}
\|x\|_{L_{p}(\mathcal{M};\ell^{r}_{2})}=\|x^{*}\|_{L_{p}(\mathcal{M};\ell^{c}_{2})}=\Big\|\Big(\sum\limits_{k=1}^{n}|x^{*}_{k}|^{2}\Big)^{1/2}\Big\|_{L_{p}(\mathcal{M})}.
\end{equation*}
The two norms above usually are not comparable at all if $p\neq 2$. Let $L_p(\M,\ell_{2}^c)$ (resp. $L_p(\M,\ell_{2}^r)$) be the completion of the space of all finite sequences in $L_p(\M)$ with respect to $\|\cdot\|_{L_p(\M,\ell_{2}^c)}$ (resp. $\|\cdot\|_{L_p(\M,\ell_{2}^r)}$).

\subsection{Noncommutative martingales}

Let $(\mathcal{M}_n)_{n\geq 1}$ be an increasing sequence of von Neumann subalgebras of  $\mathcal{M}$ such that $\bigcup_{n\geq1}\mathcal{M}_n$ is weak-$*$ dense in $\mathcal{M}.$
Let $\mathcal{E}_n$ be the  conditional expectation  from $\mathcal{M}$ onto $\mathcal{M}_n.$ A sequence $x=(x_n)_{n\geq1}$ in $L_{1}(\M)+L_\infty(\M)$ is called a noncommutative martingale with respect to $(\mathcal M_n)_{n\geq1}$ if
\[
\mathcal E_n(x_{n+1})=x_n,\qquad \forall n\geq1.
\]
The sequence $(dx_{k})_{k\geq 1}$ is the martingale differences of the martingale $x=(x_n)_{n\geq1}$, which is defined by  $dx_{1}\coloneqq x_{1}$ and
\[
dx_{k}\coloneqq x_k-x_{k-1},\qquad \forall k\geq2.
\]
If  $(x_n)_{n\geq1}\subseteq L_p(\mathcal{M})$ for some $1\leq p\leq \infty$, then $x$ is called an $L_p$-martingale. In this case, we set
$$\|x\|_{L_p(\mathcal{M})}\coloneqq\sup_{n\geq1}\|x_n\|_{L_{p}(\mathcal{M})}.$$
If $\|x\|_{L_p(\M)}<\infty$, then $x$ is called an $L_p$-bounded martingale. 

%
 

Recall that $(\M_n)_{n\geq1}$ is a regular filtration if there exists a positive number $R_{\rm reg}\geq1$ such that
\begin{equation}\label{regular}
	\mathcal{E}_n (x)\leq R_{\rm reg}\mathcal{E}_{n-1}(x)
\end{equation}
for each positive $x\in L_1(\mathcal{M})$. The number $R_{\rm reg}$ is usually called the regularity  constant.

Throughout this paper, we will mainly focus on the following concrete   regular filtration in  a given Vilenkin space.
	Let $m:=\{m_{k}\}_{k \in \mathbb{N}}$ be a sequence of positive integers such that $m_{k} \geq 2$,  and let $G_{m_{k}}$ denote a set of cardinality $m_{k}$, for each $k\in\mathbb{N}$.
We equip $G_{m_{k}}$ with the discrete topology and the uniform probability measure $\mu_{k}$ (for each $t\in G_{m_k}$).
Set 
\begin{equation}\label{eq:vs}
	G_{m}=\prod_{k\in\mathbb{N}}G_{m_k}.
\end{equation}
 Thus each $t \in G_{m}$ is a sequence $t:=\left(t_{0}, t_{1}, \ldots\right)$, where $t_{k} \in G_{m_{k}}, k \in \mathbb{N}$. Then $G_{m}$ is called a Vilenkin space. Indeed, $G_{m}$ is a compact totally disconnected space with normalized regular Borel measure $\mu$ such that  $\mu\left(G_{m}\right)=1$.  Throughout the paper, we assume that Vilenkin space $G_{m}$ is bounded, that is, $M:=\sup_km_k<\infty$. 

For each $t\in G_m$ and $1\leq n\in \mathbb{N}$, a base for the neighborhoods of $G_{m}$ is given as follows
$$I_{0}(t):=G_{m}, \quad I_{n}(t):=\left\{s=\left(s_{i}\right)_{i \in \mathbb{N}} \in G_{m}: s_{i}=t_{i} \text { for } i<n\right\}.$$
 Particularly, set $I_n:=I_n(0)$. Let 
$$\mathcal{F}:=\sigma\left\{I_{n}(t): n \in \mathbb{N}, t \in G_{m}\right\}$$
be the $\sigma$-algebra generated by all  intervals on $G_{m}$.

\begin{example}[regular filtration in Vilenkin space]\label{one-d} Consider $(G_m,\mathcal{F},\mu)$. For convenience, denote $|Q|$ the  measure of $Q$. We simply write $L_{\infty}(G_m,\mathcal{F},\mu)$ by $L_{\infty}(G_m)$. Recall that $\mathcal{N}=L_{\infty}(G_m)\bar{\otimes} \mathcal{M}$ is the tensor von Neumann algebra equipped with the trace $\varphi=\int \otimes \tau$.
Set 
\begin{equation}\label{df}
\mathcal{F}_n =\sigma \{ I_n(t): t\in G_m\}, \quad n\geq1.
\end{equation}
Then $(\mathcal{F}_n)_{n\geq1}$ is the filtration of $\mathcal{F}$. The conditional expectations with respect to $(\mathcal{F}_n)_{n\geq1}$ are denoted by $(\mathbb{E}_n)_{n\geq 1}$. More precisely, for $f\in L_1(G_m)$,
\begin{equation}\label{ce}
\mathbb{E}_n(f)=\sum_{Q\in D(\mathcal{F}_n)}  f_Q\chi_Q, \quad \forall n\geq1,
\end{equation}
where 
$D(\mathcal{F}_n)$ denotes the collection of all  intervals in $\mathcal{F}_n$ with length $M_n^{-1}$ and
	
$$f_Q=\frac{1}{|Q|}\int_{Q} f(t)d\mu(t).$$
For each $n$, we simply write $L_{\infty}(G_m,\mathcal{F}_n,\mu)$ by $L_{\infty}(\mathcal{F}_n)$ and let

$$\mathcal{N}_n=L_{\infty}(\mathcal{F}_n)\bar{\otimes} \mathcal{ M}.$$
Then $(\mathcal{N}_n)_{n\geq 1}$ forms an increasing sequence of von Neumann subalgebras of  $\mathcal{N}$ such that $\bigcup_{n\geq1}\mathcal{N}_n$ is weak-$*$ dense in $\mathcal{N}.$ The conditional expectation of $\mathcal{N}$ onto $\mathcal{N}_n$ is $\mathbb{E}_n\otimes \mathrm{id}_{\M}$ (simply $\mathbb{E}_n$),
where $\mathrm{id}_{\M}$ is the indentity map from $\M$ onto $\M$. Moreover, it is easy to see that the filtration  $(\mathcal{N}_n)_{n\geq 1}$ is regular with   $R_{\rm reg}=\sup_km_k$ in the sense of \eqref{regular}.

\end{example}



\subsection{Noncommutative Calder\'{o}n-Zygmund decomposition} 
In this paper, our proof mainly depends on the
noncommutative Calder\'{o}n-Zygmund decomposition established in \cite{CCP2022}. To state the  decomposition, we first recall the Cuculescu projections. 

\begin{lemma}[{\cite{Cuc} or \cite[Proposition 2.3]{Rand2002}}]\label{Cuculescu}
	For positive $f\in L_{1}(\M)$ and $\lambda\in \R^{+}$, let $(f_n)_{n\geq1}$ be the martingale defined by $f_n=\mathcal{E}_n(f)$ for each $n\geq1$. Then there exists a sequence of decreasing projections  $(q_n^{(\lambda)})_{n\geq 1}$ in $\mathcal{M}$ satisfying the following properties:
	\begin{enumerate}[{\rm (i)}]
		\item for every $n\geq1 $, $q_n^{(\lambda)}\in \mathcal{M}_n $;
		\item for every $n\geq1 $, $q_n^{(\lambda)}$ commutes with $q_{n-1}^{(\lambda)}f_n q_{n-1}^{(\lambda)}$;
		\item for every $n\geq1 $, $q_n^{(\lambda)}f_nq_n^{(\lambda)}\leq \lambda q_n^{(\lambda)}$;
		\item if we set $q^{(\lambda)}= \wedge_{n=1}^{\infty} q_n^{(\lambda)}$, then $q^{(\lambda)}fq^{(\lambda)}\leq \lambda q^{(\lambda)}$ and
		\[
		\lambda\tau({\bf 1}-q^{(\lambda)})\leq \tau\big(({\bf 1}-q^{(\lambda)})f\big)\leq \|f\|_{L_1(\mathcal{M})}.
		\]
	\end{enumerate}
\end{lemma}

In what follows  we will simply write $(q_n)_{n\geq 1}$   for the sequence of Cuculescu's projections $(q_n^{(\lambda)})_{n\geq 1}$  associated to the martingale generated by  a given positive $f \in L_1(\mathcal{M})$ and $\lambda\in \R^{+}$, and let $q= q^{(\lambda)}$ for convenience. Set $p_{1}= \1-q_{1}$ and
\begin{equation}\label{pn}
	p_{n}= q_{n-1}-q_{n}, \quad n\geq 2.
\end{equation} 
Then
\begin{equation}\label{sump}
	\sum_{n\geq1} p_n={\bf 1}-q,
\end{equation}
where the convergence is taken with respect to the strong operator topology.

The noncommutative Calder\'{o}n-Zygmund decomposition was first established in \cite{Pa2009}. The great advantage of the following improved decomposition enables one to circumvent the pseudo-localization technique. 

\begin{theorem}[{\cite[Lemma 1.1]{CCP2022}}]\label{NewCZ}
	Let $f$ be a positive element in $L_1(\mathcal{M})$ and $\lambda>0$. Suppose that the associated filtration $(\mathcal{M}_n)_{n\geq1}$ is regular with  $R_{\rm reg}\geq1$ in the sense of \eqref{regular}. The sequence of projections  $(q_{k})_{k\geq 1}$ and $(p_{k})_{k\geq 1}$ are as in Lemma \ref{Cuculescu} and \eqref{pn}. Then there is a decomposition of $f$ such that
	\begin{equation}\label{dec}
		f=g+b_{d}+b_{off}
	\end{equation}
	where
	\begin{enumerate}[\rm (i)]
		\item $g=qfq+\sum_{k\geq1} p_kf_kp_k$,  $$\|g\|_{L_1(\mathcal{M})}\leq \|f\|_{L_1(\mathcal{M})}\quad \mbox{and}\quad \|g\|_{L_{\infty}(\mathcal{M})}\leq R_{\rm reg}\lambda;$$
		\item $b_d=\sum_{k\geq1} p_k(f-f_k)p_k$ and $$\sum_{k\geq1}\|p_k(f-f_k)p_k\|_{L_1(\mathcal{M})}\leq 2\tau(f({\bf 1}-q))\leq 2\|f\|_{L_1(\mathcal{M})};$$
		\item $b_{off}=\sum_{k\geq1} p_k(f-f_k)q_k+q_k(f-f_k)p_k$. 
	\end{enumerate}
\end{theorem}

Consider Theorem \ref{NewCZ} with the filtration $(\mathcal{N},\varphi; (\mathcal{N}_{n})_{n\geq1})$ as in Example \ref{one-d} and positive $f\in L_1(\mathcal{N})$.  We  can then express the projections $q_n$ and $p_n$ as 
\begin{equation}\label{q}
	q_n=\sum_{Q\in D(\mathcal{F}_n)} q_Q\chi_Q,
\end{equation}
where $q_Q\in\mathcal{P}(\M)$, and 
\begin{equation}\label{pe}
	p_n=\sum_{Q\in D(\mathcal{F}_{n-1})} q_Q\chi_Q- \sum_{Q\in D(\mathcal{F}_{n})} q_Q\chi_Q=\sum_{Q\in D(\mathcal{F}_{n})}( q_{\widehat{Q}}-q_Q)\chi_Q=\sum_{Q\in D(\mathcal{F}_{n})}p_Q\chi_Q,
\end{equation}
where, for $Q\in D(\mathcal{F}_n)$, $\widehat{Q}$ is the  unique  interval in $ D(\mathcal{F}_{n-1})$ such that $Q\subseteq \widehat{Q}$.
One can  observe that for each $n\geq1$ and $Q\in D(\mathcal{F}_{n})$,
\begin{equation}\label{pq}
	p_Qq_Q=q_Qp_Q=0.
\end{equation}
Based on \eqref{pe}, we can rewrite $b_d$ and $b_{off}$ as follows
\begin{eqnarray}\label{bd}
	b_d=\sum_{k\geq1} \sum_{Q\in D(\mathcal{F}_k)} p_Q(f-f_Q)p_Q\chi_Q
\end{eqnarray}
and 
\begin{eqnarray}\label{boff}
	b_{off}=\sum_{k\geq1} \sum_{Q\in D(\mathcal{F}_{k})} p_Q(f-f_Q)q_Q\chi_Q+q_Q(f-f_Q)p_Q\chi_Q.
\end{eqnarray}

\subsection{Vilenkin-like  system and Ces\`{a}ro means for operator-valued functions}
In this subsection, we introduce the Vilenkin-like  system on $G_m$ and refer to  \cite{Ga2001} for more details. 
Let $M_0:=1$, and $M_{n+1}:=m_nM_n$, $n\in\mathbb{N}$. Then each $n\in\mathbb{N}$ can be uniquely expressed as
\begin{align}\label{n-V-L}
n=\sum_{k=0}^\infty n_kM_k,\quad 0\leq n_k\leq m_k-1, k\in\mathbb{N}.
\end{align}
Denote 
\begin{equation}\label{n^(k)-V-L-}
n^{(k)}=\sum_{j=k}^\infty n_jM_j.
\end{equation}
If $n=\sum_{j=k}^\infty n_jM_j$, we call that $M_k$ is a divisor of $n$.


\begin{assumption}\label{gamma-nk}  Let  complex-valued functions $r_k^n:G_m\rightarrow \mathbb{C}$  satisfy the following.
\begin{enumerate}[{\rm (i)}]
	\item  For each $k\in\mathbb{N}$, $(r_k^n)_{n\in\mathbb{N}}$ are $\mathcal{F}_{k+1}$-measurable, and $r_k^0=1$.
	\item Let $k, l, n \in \mathbb{N}$ be such that $n^{(k+1)}=l^{(k+1)}$. If $M_{k}$ is a divisor of both $n$ and $l$,  then 
	$$
\mathbb{E}_{k}\left(r_{k}^{n} \bar{r}_{k}^{l}\right) = \left\{
\begin{array}{ll}
1, & \hbox{if $n_{k}=l_{k}$} \\
0, & \hbox{if $n_{k} \neq l_{k}$}
\end{array}. 
\right.
$$
Here, the conditional expectation $\mathbb{E}_k$ is defined as in \eqref{ce}.
	\item If $M_{k+1}$ is a divisor of $n$, then
$$\sum_{j=0}^{m_{k}-1}\left|r_{k}^{j M_{k}+n}(t)\right|^{2}=m_{k}, \quad \mbox{for all $t \in G_{m}$}.$$
     \item There exists a constant $\delta>1$ for which $\left\|r_{k}^{n}\right\|_{\infty} \leq \sqrt{m_{k} / \delta}$.
\end{enumerate}
\end{assumption}

The product system generated by the complex-valued functions $(r_k^n)_{k,n}$ is the Vilenkin-like system: $\psi_0=1$ and
\begin{align}\label{WPS}
\psi_{n}:=\prod_{k=0}^{\infty} r_{k}^{n^{(k)}}, \quad n \in \mathbb{N},
\end{align}
where $n=\sum_{k=0}^{\infty} n_{k} M_k$.
Since $r_{k}^{0}=1$, for $M_a\leq n<M_{a+1}$, we have
$\psi_{n}=\prod_{k=0}^{a} r_{k}^{n^{(k)}}$.

Here, we provide concrete examples of Vilenkin-like system, and refer the reader to \cite{Ga2001} for more examples. We omit the   elementary computations that the given  examples satisfy Assumption \ref{gamma-nk}.

\begin{example}[Vilenkin and Walsh systems]\label{V-W-S-example} Consider Vilenkin space $G_m$. Moreover, for each $k$, let $G_{m_k}:=\mathbb{Z}_{m_{k}}$ be  the $m_{k}$-th  discrete cyclic group. That is $\mathbb{Z}_{m_{k}}$ can be represented by the set $\left\{0,1, \ldots, m_{k}-1\right\}$, where the group operation is the $\bmod\ m_{k}$ addition and every subset is open. The group operation on $G_{m}$ is the coordinate-wise addition. Then $G_{m}$ is called a Vilenkin group.  Let 
$$r_{k}^{n}(t):=\left(\exp \left(\frac{2 \pi i t_{k}}{m_{k}}\right)\right)^{n_{k}},\quad t\in G_{m},$$ where $i:=\sqrt{-1}$. The system $\psi:=\left(\psi_{n}\right)_{n \in \mathbb{N}}$ is the Vilenkin system.  If $m_{k}=2$ for all $k \in \mathbb{N}$, we get the Walsh(-Paley) group, and the related system is the Walsh(-Paley) system.  For more information about Vilenkin and Walsh   group are referred to \cite{Golubov1991,PTW2022,SWS1990}.
\end{example}

\begin{example}[the group of $m$-adic (2-adic) integers]\label{m-adic-example}  Let $G_{m_{k}}:=\{0,1, \ldots, m_{k}-1\}$ for all $k \in \mathbb{N}$. Define on $G_{m}$ the following (commutative) addition: if $s, t \in$ $G_{m}$, then $s+t=v \in G_{m}$ is defined recursively as follows: $s_{0}+t_{0}=q_{0} m_{0}+v_{0}$, where  $v_{0} \in\left\{0,1, \ldots, m_{0}-1\right\}$ and $q_{0} \in \mathbb{N}$. Suppose that $v_{0}, \ldots, v_{k}$ and $q_{0}, \ldots, q_{k}$ have been defined. Then write $s_{k+1}+t_{k+1}+q_{k}=q_{k+1} m_{k+1}+v_{k+1}$, where $v_{k+1} \in\left\{0,1, \ldots, m_{k+1}-1\right\}$ and $q_{k+1} \in \mathbb{N}$. Then $G_{m}$ is called the group of $m$-adic integers (if $m_{k}=2$ for all $k \in \mathbb{N}$, it is the group of 2-adic integers). In this case let

$$
r_{k}^{n}(t):=\left(\exp \left(2 \pi \imath\left(\frac{t_{k}}{m_{k}}+\frac{t_{k-1}}{m_{k} m_{k-1}}+\ldots+\frac{t_{0}}{m_{k} m_{k-1} \ldots m_{0}}\right)\right)\right)^{n_{k}}.
$$
Then the system $\psi:=\left(\psi_{n}\right)_{n \in \mathbb{N}}$ is the character system of the group of $m$-adic integers. For more  information about  the group of $m$-adic integers we refer the reader to \cite{HR1963, Ta1975}.
\end{example}


 Next, we  introduce Ces\`{a}ro means of Vilenkin-like-Fourier series for operator-valued functions.
Recall that $\mathcal{N}=L_{\infty}(G_m)\bar{\otimes}\mathcal{M}$, where $\mathcal{M}$ is a semifinite von Neumann algebra.   The $n$-th Vilenkin-like  Fourier coefficient of $f\in L_1(\mathcal{N})$ is defined by

$$\widehat{f}(n):=\int_{G_m} f(t) \bar{\psi}_{n}(t) d \mu(t),\quad \mbox{$n\in \mathbb{N}$}.$$
Denote by $S_{n}(f)$ the $n$-th partial sum of the Vilenkin-like Fourier series of  $f\in L_1(\mathcal{N})$, namely,
\begin{equation}\label{ps}
S_{n}(f)\coloneqq\sum_{k=0}^{n-1} \widehat {f}(k)\psi_{k},\quad n\in \mathbb{N}.
\end{equation}
We can rewrite the partial sum  as the  following  form 
$$
S_{n}(f)(\eta) =\int_{G_m} f(t) D_{n}(\eta,t) d\mu(t), \qquad n\in \mathbb{N}, \eta\in G_m,
$$
where $D_n$ is the  Vilenkin-like-Dirichlet kernel defined by
$$
D_n(\eta,t) := \sum_{k=0}^{n-1} \psi_k(\eta)\bar{\psi}_k(t)
$$
satisfying (see \cite[Lemma 3]{Ga2001})
\begin{equation}\label{e5}
D_{M_n}(\eta,t) = \left\{
\begin{array}{ll}
M_n, & \hbox{if $\eta \in I_n(t)$} \\
0, & \hbox{if $\eta \in G_m\setminus I_n(t)$}
\end{array}, \qquad   n\in \mathbb N.
\right.
\end{equation}
From \eqref{e5}, it is easy to see that
\begin{equation}\label{S2n}
S_{M_n}(f)=\mathbb{E}_n(f)=\sum\limits_{j=0}^{M_n-1}\widehat{f}(j)\psi_{j}, \qquad n\in \mathbb N.
\end{equation}
The Ces\`{a}ro means  of $f\in L_1(\mathcal{N})$ are defined  by 
\begin{equation}\label{CM1}
\sigma_{n}(f):= \frac{1}{n} \sum_{k=1}^{n} S_{k}(f)=\sum_{k=0}^{n-1}\left(1-\frac{k}{n}\right)\widehat{f}(k)\psi_k,\quad n\in\mathbb{N}.
\end{equation}
It is straightforward that 
\begin{equation}\label{CM}
\sigma_{n}(f)(\eta)=\int_{G_m} f(t) K_n(\eta,t)  d\mu(t), \quad \eta\in G_m,\quad n\in\mathbb{N},
\end{equation}
where $K_n$ is the  Fej\'er kernel  defined by
$
K_n := \frac{1}{n} \sum_{k=1}^n D_k, \qquad n\in\mathbb{N}.
$

Denote 
 $$K_{a,b}=\sum_{k=a}^{a+b-1}D_k\quad \mbox{and}\quad  \psi_{n,k}:=\prod_{s=k}^\infty r_s^{n^{(s)}},$$
where $a,b,n,k\in\mathbb{N}$.

The following lemmas provide  some essential properties of the Fej\'er kernel for Vilenkin-like system; see for instance   \cite[Lemma 9, Lemma 10, and Proposition 11]{Ga2001}. In the sequel of the present paper, the number $\delta$ is always taken from Assumption \ref{gamma-nk}.
\begin{lemma}\label{K-E-ab-1} Let $n,a,b\in \mathbb{N}$. If  $M_n\leq l<M_{n+1}$ and  $b\leq a\leq n$ for each $t\in I_a(s)\setminus I_{a+1}(s)$, then there is a constant $c>0$ such that
$$\big|K_{l^{(b+1)}+jM_b,M_b}(t,s)\big|\leq c\delta^{a-n}M_bM_n,\quad 0\leq j\leq m_b-2.$$
\end{lemma}

\begin{lemma}\label{K-E-ab-2} Let $n,a,b\in \mathbb{N}$. If $M_n\leq  l<M_{n+1}$, $a<b\leq n$ and $0\leq j\leq m_b-2$, then, for each $s \in G_{m}$,  there exists   a constant $c>0$ such that
$$\int_{I_{a}(s) \backslash I_{a+1}(s)}\left|K_{l^{(b+1)}+j M_{b}, M_{b}}(t,s)\right|^{2} d \mu(t) \leq c \delta^{a-n} M_{a} M_{b} M_{n} .
$$
\end{lemma}

\begin{lemma}\label{K-E-GI} 
Let  $k \in \mathbb{N}$. Then, for each $s \in G_{m}$,  there exists a constant $c>0$ such that
$$\int_{G_{m} \backslash I_{k}(s)} \sup _{n \geq M_{k}}\left|K_{n}(t,s)\right| d \mu(t) \leq c.$$
\end{lemma}

\subsection{Two useful lemmas}
We conclude this section by two lemmas which will be used in the sequel.

%
%
%

%

\begin{lemma}\label{AcB}
For any $a,b\in L_2(\mathcal{N})$, there exists $u\in L_{\infty}(\M)$ with $\|u\|_{L_{\infty}(\M)}\leq 1$ such that
\[
\int a^{*}(t)b(t)dt=\Big(\int a^{*}(t)a(t)dt\Big)^{1/2} \cdot u\cdot\Big(\int b^{*}(t)b(t)dt\Big)^{1/2}.
\]
\end{lemma}

\begin{proof}
By approximation, we may assume that
\[
A=\int a^{*}(t)a(t)dt,\quad B=\int b^{*}(t)b(t)dt
\]
are both invertible. Set
\[
u=A^{-1/2}\cdot \int a^*(t)b(t)dt \cdot B^{-1/2}.
\]
It is obvious that
\[
\int a^*(t)b(t)dt=A^{1/2} u B^{1/2}.
\]
It now remains to show $\|u\|_{L_{\infty}(\M)}\leq 1$. Applying \cite[Proposition 1.1]{Me2007} with $r=p=q=\infty$, we get 
\begin{align*}
\|u\|_{L_{\infty}(\M)}&\leq  \Big\|\Big(\int A^{-1/2}a^*(t)a(t)A^{-1/2}dt\Big)^{1/2}\Big\|_{L_{\infty}(\mathcal{M})}\\
&\quad \times\Big\|\Big(\int B^{-1/2}b^*(t)b(t)B^{-1/2}dt\Big)^{1/2}\Big\|_{L_{\infty}(\mathcal{M})}\\
 &= \| A^{-1/2}AA^{-1/2}\|_{L_{\infty}(\mathcal{M})}^{1/2} \| B^{-1/2}BB^{-1/2}\|_{L_{\infty}(\mathcal{M})}^{1/2}= 1.
\end{align*}
\end{proof}

The following lemma is useful to prove Theorem \ref{main-asy}. For $2\leq p<\infty$, the right hand side of the below inequality was proved in \cite[Section 2]{HWW2019} via another method.
\begin{lemma}\label{ebm-pc}
Let $1\leq p<\infty$. Then, for any $x=(x_k)_{k\geq1}\in L_p(\M,\ell_2^c)$, we have
\[
\|x\|_{\Lambda_{p,\infty}(\M,\ell_{\infty}^c)}\leq \|x\|_{L_p(\M, \ell_{\infty}^c)}\leq \|x\|_{L_p(\M,\ell_{2}^c)}.
\]
The same result holds true for corresponding row spaces. 
\end{lemma}

\begin{proof}
 According to the definition of $\|\cdot\|_{L_p(\M,\ell_{\infty}^c)}$, for any $\varepsilon>0$, there exists  $A\in L_p(\M)$ such that for each $k\geq1$,
	$$x_k=v_kA,\quad \mbox{and}\quad \|A\|_{L_p(\M)}\leq \|x\|_{L_p(\M,\ell_{\infty}^c)}+\varepsilon,$$
	where $\sup_{k\geq1}\|v_k\|_{L_\infty(\M)}\leq 1$. 
	It is immediate that 
	$$|x_k|^2=A^*v_k^*v_kA\leq |A|^2, \quad \forall k\geq1.$$
	For any $\lambda>0$, take 
	$e=\chi_{(0,\lambda^2)}(|A|^2).$
	Then, for all $k\geq1$, 
	$$\|x_ke\|_{L_\infty(\M)}=\|ex_k^*x_ke\|_{L_\infty(\M)}^{1/2}\leq \|e|A|^2e\|_{L_\infty(\M)}^{1/2}\leq \lambda.$$
	On the other hand, by the Chebyshev inequality, we have
	$$\lambda\tau(1-e)^{1/p}\leq \lambda \left(\frac{\||A|^2\|_{L_{\frac{p}{2}}(\M)}^{p/2}}{\lambda^p}\right)^{1/p}=\|A\|_{L_p(\M)}\leq \|x\|_{L_p(\M,\ell_{\infty}^c)}+\varepsilon.$$
The first inequality follows by letting $\varepsilon$ tend to zero. 
	
For the second inequality, we first note that $|x_k|^2\leq \sum_{k\geq1}|x_k|^2$. Let $S=(\sum_{k\geq1}|x_k|^2)^{1/2}$. Then there are contractions $(u_k)_{k\geq1}\subseteq L_{\infty}(\M)$ such that 
	\[
	|x_k|=u_kS, \quad \forall k\geq1.
	\]
	From the polar decomposition, we can find contractions $(v_k)_{k\geq1} \subset L_{\infty}(\M)$ such that
	\[
	x_k=v_k u_k S,\quad \forall k\geq1.
	\]
	Thus, according to the definition of $\|\cdot\|_{L_p(\M,\ell_{\infty}^c)}$, 
	$$\|x\|_{L_p(\M,\ell_{\infty}^c)}\leq \|S\|_{L_p(\M)}=\|x\|_{L_p(\M,\ell_{2}^c)}.$$
\end{proof}

\section{Noncommutative Vilenkin system and transference argument}\label{noncommutative Walsh and transference}
In this section, we introduce the noncommutative Vilenkin system (following \cite{DS2000,SS2018}) and establish the so-called transference argument with respect to this system. Such transference technique allows one to convert situations from operator-valued setting to general noncommutative setting. The transference argument stated in this section maybe well known for the experts but we prefer to write it down and provide the proof in the appendix. 

Denote by $\mathbb{M}_n$ the space of $n \times n$ complex valued matrices with the distinguished normalised trace $tr_{n}$, that is, ${\rm tr}_{n}\left(\mathbf{1}_{n}\right)=1$, where $\mathbf{1}_{n}$ is the $n \times n$ identity matrix.  We identify $\mathcal{R}$ with the relative infinite tensor product as follows
$$(\mathcal{R}, \tau)=\bigotimes_{k=0}^{\infty}(\mathbb{M}_{m_k}, {\rm tr}_{m_k}).$$
Note that such $\tau $ is a faithful normal trace on $\mathcal R$ with $\tau(\1)=1$, where $\1\coloneqq\otimes_{j=1}^{\infty}\1_{2}$. For each $n\in \Na$, consider the von Neumann algebra given by
\[
(\mathcal R_n,\tau_n)=\bigotimes_{k=0}^{n-1} (\mathbb M_{m_k}, {\rm tr}_{m_k}).
\]
It is cannonical to view $\mathcal R_n$ as a von Neumann subalgebra of $\mathcal R_{n+1}$ (resp. $\mathcal R$) via the inclusion
\[
x\in\Ra_{n}\to x\otimes \left(\bigotimes\limits_{k=n}^{\infty}\1_{m_k}\right)\in \Ra.
\]
It is clear that the union $\bigcup_{n\geq1} \mathcal{R}_n$ is clearly weak-$*$ dense in $\mathcal{R}$.

Given a  sequence $m=(m_k)_{k\geq1}$, the Vilenkin group $G_{m}$ is   defined as the product
$$G_{m}=\prod_{k=1}^{\infty} \mathbb{Z}_{m_k}$$
endowed with the product topology, where $\mathbb{Z}_{m_k}$ is the cyclic group of order $m_k$. All Vilenkin groups we work with will be of this form. The dual group, $\widehat{G}_{m}$ is given by

$$\widehat{G}_{m}=\coprod_{k=1}^{\infty} \widehat{\mathbb{Z}}_{m_k},$$
where $\widehat{\mathbb{Z}}_{m_k}$ is the dual group of, and isomorphic to, $\mathbb{Z}_{m_k}$. 
The following details the construction of the corresponding Vilenkin system in $\mathcal{R}$. For each integer $k \geq 1$, let $A_{k}$ and $B_{k}$ denote the matrices given by

$$
A_{k}=\sum_{j \in \mathbb{Z}_{m_k}} \exp \left(\frac{2 \pi \imath j}{m_k}\right) e_{j, j}^{m_k}, \quad B_{k}=\sum_{j \in \mathbb{Z}_{m_k}} e_{j, j+1}^{m_k},
$$
where for each $n \geq 1, e_{p, q}^{n}$ is the $n \times n$ unit matrix defined such that $\left(e_{p, q}^{n}\right)_{p, q}=$ 1, and all other entries are zero, with entries indexed by $\mathbb{Z}_{n}^{2}$.
Within $\mathcal{R}$, the elements of the Vilenkin system are enumerated by doubled sequences in $\widehat{G}_{m}$.  We shall identify the group $G_m\times G_m$ (respectively, $\widehat{G}_{m}\times \widehat{G}_{m}$) with the Vilenkin group $G_{2m}$ (respectively, $\widehat{G}_{2m}$) where
\begin{equation}\label{2m}
2m:=\{(2m)_k\}_{k\geq0},
\end{equation}
where $(2m)_{2k}=(2m)_{2k+1}=m_k$ for each $k\in\mathbb{N}.$
It follows that any $\eta\in \widehat{G}_{2m}$ can be considered as a pair $(\eta',\eta'')$ where
$$\eta':=(\eta_{2k})_{k\geq0}\quad \mbox{and \quad $\eta'':=(\eta_{2k+1})_{k\geq0}$}$$
are elements of $\widehat{G}_m$. For each $\eta \in \widehat{G}_{2 m}$, define
\begin{align}\label{non-V-S}
W_{\eta}=\bigotimes_{k=0}^{\infty} A_{k}^{\eta_{2 k}} B_{k}^{\eta_{2 k+1}}.
\end{align}
The Vilenkin system associated to $G_{m}$ in $\mathcal{R}$ is then given by the set $\{W_{\eta}\}_{\eta \in \widehat{G}_{2 m}}$. For doubled sequences, let $G_{2 N}$ denote the subgroup of sequences of length $2 N$ in $G_{2 m}$, and $\widehat{G}_{2 N}$ denote the subgroup of sequences of length $2 N$ in $\widehat{G}_{2 m}$. It is then easily verified that for each $N \geq 1$, the restriction of the Vilenkin system to $\{W_{\eta}\}_{\eta \in \widehat{G}_{2 N}}$ forms a basis for $\mathcal{R}_{N}$ and $L_{2}(\mathcal{R}_{N})$.  Since the mapping 
$$k:\eta \mapsto k_{\eta}$$
is an order isomorphism of $\widehat{G}_{2m}$  and $\mathbb{N}=\{0,1,2, \ldots\}$. For each $n\in\mathbb{N}$, there exists a element $\eta\in G_{2m}$ such that  $$k_\eta=\sum_{k=0}^\infty\eta_kM_k=n,$$
where $M_0=1$, $M_{1}=(2m)_0M_0$ and $M_{k}=(2m)_{k-1}M_{k-1}$ for each $k\geq2$.
Denote $W_n:=W_{k^{-1}(n)}$. Then the Vilenkin system $\{W_n\}_{0\leq n \leq M_{2N}-1}$ also forms an orthonormal basis for $\mathcal{R}_{N}$ and $L_{2}\left(\mathcal{R}_{N}\right)$.

For 
\begin{equation*}
n=\sum_{k=0}^{\infty} n_k M_k, \quad m=\sum_{k=0}^{\infty} m_k M_k,\qquad n_k,m_k\in\mathbb{Z}_{(2m)_k},
\end{equation*}
let 
\begin{equation}\label{dotplus}
	m  \bigtriangleup n= \sum_{k=0}^{\infty}(m_k\dot{+} n_k)M_k.
\end{equation}
where $m_k\dot{+} n_k=(n_k+m_k)\mod(2m)_k$.
For  $t,s\in G_{2m}$, the  addition is defined in the following way: 
$$t\dot{+} s:=\{t_k\dot{+} s_k\}_{k
\geq0}.$$
Recall that (see e.g. \cite[p.\,114]{We2002})
\begin{equation}\label{w-plus}
\psi_m(t\dot{+} s)=\psi_m(s)\psi_m(t),\qquad	\psi_m\psi_n=\psi_{m\bigtriangleup n}.
\end{equation}
The following lemma is taken from \cite[Lemma 2.1]{SS2018}. 
\begin{lemma}\label{RW-plus} For any $m,n\in\mathbb{N}$, there exist complex numbers   $u_n$ and $\omega_{m,n}$ such that $W_n^*=u_nW_{-n}$ and $W_mW_n=\omega_{m,n}W_{m\bigtriangleup n}$, where $|u_n|=1$, $|\omega_{m,n}|=1$ and $m\bigtriangleup n $ is referred to \eqref{dotplus}.
\end{lemma}

The noncommutative Vilenkin-Fourier coefficient of $x\in L_1(\mathcal{R})$  is defined by $\widehat{x}(0)=\tau(x)$ and
\[
\widehat{x}(k)=\tau(x\cdot W_k^*), \quad k\in \mathbb{N}.
\]
Analogous to \eqref{ps} and \eqref{CM1},  for $x\in L_1(\mathcal{R})$, we define the partial sum and  Ces\`{a}ro means respectively as follows
\[
S^{\mathcal{R}}_n(x)\coloneqq\sum_{k=0}^{n-1}\widehat{x}(k)W_k,\quad n\in \mathbb{N}
\]
and
\begin{equation}\label{sigma-R}
	\sigma^{\mathcal{R}}_{n}(x)\coloneqq\frac{1}{n} \sum_{k=1}^{n} S^{\mathcal{R}}_{k}(x)=\sum_{k=0}^{n-1}\left(1-\frac{k}{n}\right)\widehat{x}(k)W_k,\quad n\in\mathbb{N}.
\end{equation}

Let $(W_{k})_{k=0}^{\infty}$ stand for the noncommutative Vilenkin system from now on.
A polynomial  in $\mathcal{R}$ is a finite sum 
$$x=\sum_{k\geq0}\alpha_kW_k, \quad \alpha_k\in \mathbb{C},$$
that is $\alpha_k=0$ for all but finite indices $k\in\mathbb{N}$. Let $\mathrm{Poly}(\Ra)$ be  the $*$-subalgebra of $\mathcal{R}$ consisting with all (complex) polynomials in $\mathcal{R}$. For each $k\in \mathbb{N}$, let  $\Ex_{k}:\mathcal{R}\to \mathcal{R}_{k}$ be the conditional expectation. For each $k\in \Na$, we then have
\begin{equation}\label{EkSk}
\Ex_{k}(x)= \sum\limits_{j=0}^{M_{2k}-1}\widehat{x}(j)W_{j}=S^{\Ra}_{M_{2k}}(x),\quad x\in \mathrm{Poly}(\Ra).
\end{equation}
In the sequel of the paper, equip the von Neumann algebra  $L_{\infty}(G_{2m})\bar{\otimes} \mathcal{R}$ the trace $\int\otimes \tau$ which is still denoted by $\varphi$.  Define the map $\gamma:\mathrm{Poly}(\Ra)\to L_{\infty}(G_{2m})\bar{\otimes} \mathcal{R}$ by setting
\begin{equation}\label{transference mapping}
 \gamma(x)=\sum_{j\geq0}\widehat{x}(j)\psi_{j}\otimes W_{j}, \quad x\in \mathrm{Poly}(\Ra),
\end{equation}
where  $(\psi_j)_{j\geq0}$ are the classical Vilenkin functions as in Example \ref{V-W-S-example} with respect to the Vilenkin group $G_{2m}$.

The results below should be known. We also provide detail proofs in the appendix.

\begin{proposition}\label{transference 1}
The mapping $\gamma:\mathrm{Poly}(\Ra)\to L_{\infty}(G_{2m})\bar{\otimes}\mathcal{R}$ satisfies the following properties:
\begin{enumerate}[{\rm (i)}]
\item $\gamma$ is a $*$-homomorphism;
\item $\gamma$ is trace preserving, that is, $\varphi(\gamma(x))=\tau(x)$, for all $x\in \mathrm{Poly}(\Ra)$;
\item $\gamma$ is injective, and both $\gamma$ and $\gamma^{-1}$ are normal.
\end{enumerate}
\end{proposition}

Since $\mathrm{Poly}(\Ra)$ is weak-$\ast$ dense in $ \Ra $ and $\gamma:\mathrm{Poly}(\Ra)\to L_{\infty}(G_{2m})\bar{\otimes} \mathcal{R}$ is normal, we could now extend $\gamma$ in a natural way making $\gamma$ a $*$-homomorphism from $\mathcal{R}$ into $L_{\infty}(G_{2m})\bar{\otimes} \mathcal{R}$. For simplicity, we still use the notation $\gamma$ to stand for its natural extension if no confusion arises.

\begin{proposition}\label{transference 1-2}
The mapping $\gamma:\mathcal{R}\to L_{\infty}(G_{2m})\bar{\otimes} \mathcal{R}$ is a $*$-homomorphism such that $\gamma(\mathcal{R})$ is weak-$\ast$ closed in $L_{\infty}(G_{2m})\bar{\otimes} \mathcal{R}$. In particular, $\gamma(\mathcal{R})$ is a von Neumann subalgebra of $L_{\infty}(G_{2m})\bar{\otimes} \mathcal{R}$.
\end{proposition}

With Proposition \ref{transference 1-2} at hand, we obtain the following transference results.

\begin{theorem}\label{transference argument}
For any $0<p\leq \infty$ and $ x\in\mathcal{R}$, we have $$\|\gamma(x)\|_{L_{p}(L_{\infty}(G_{2m})\bar{\otimes} \mathcal{R})}=\|x\|_{L_{p}(\mathcal{R})}.$$ Furthermore, for any $x\in \mathcal{R},$ one has the following stronger result
\begin{equation}\label{mu-est}
\mu(t,\gamma (x))= \mu(t, x),\quad \forall t>0.
\end{equation}
\end{theorem}

\section{Proofs of Theorems \ref{CS} and \ref{Nc-W}}\label{sec-pf-NcW}

This section contains the detailed proofs of Theorem \ref{CS} and Theorem \ref{Nc-W}. More precisely, in Section \ref{sec-operator}, we firstly reduce Theorem \ref{CS} to Theorem \ref{sigma-tilde}, and then prove Theorem \ref{sigma-tilde} by the noncommutative Calder\'{o}n-Zygmund decomposition given in \cite{CCP2022}. In Section \ref{sec-pf-Nc}, we will prove Theorem \ref{Nc-W} via the transference technique. 

\subsection{Ces\`{a}ro  summability in operator-valued setting}\label{sec-operator}

Our Theorem \ref{CS} can be deduced from the following stronger result, i.e. Theorem \ref{sigma-tilde}, where $\widetilde{\sigma}_{n}$ is defined by 
\begin{equation}\label{s-tilde}
\widetilde{\sigma}_{n}(f)(\eta)=\int_{G_m}f(t)\sup_{M_{n}\leq l<M_{n+1}} |K_l(\eta, t)|d\mu(t),\quad  \eta\in G_m.
\end{equation}

Let $\widetilde{\sigma}_n$ be the symbol given as in \eqref{s-tilde}, and suppose that $f\geq0$. Then, for each $n\geq1$, we have $\widetilde{\sigma}_n(f)\geq0$. For $t\in G_m$, we write
\begin{equation*}
\begin{aligned}
	\sigma_l(f)(t)&=\int_{G_m} f(s) K_{l}(t,s) d \mu(s)\\
	&=\int_{G_m} f(s) (\Re K_{l}(t,s)+i\Im K_{l}(t,s)) d \mu(s)\\
	&=\Re\sigma_l(f)(t)+i\Im\sigma_l(f)(t).
\end{aligned}
\end{equation*}
Then
$$\|(\sigma_l(f))_{l\geq1}\|_{\Lambda_{1,\infty}(\mathcal{N},\ell_{\infty})}\leq2\|(\Re\sigma_l(f))_{l\geq1}\|_{\Lambda_{1,\infty}(\mathcal{N},\ell_{\infty})}+2\|(\Im\sigma_l(f))_{l\geq1}\|_{\Lambda_{1,\infty}(\mathcal{N},\ell_{\infty})}.$$
Since
$$-\widetilde{\sigma}_n(f)\leq \Re\sigma_l(f)\leq \widetilde{\sigma}_n(f),\quad M_n\leq l<M_{n+1}$$
and
$$-\widetilde{\sigma}_n(f)\leq \Im\sigma_l(f)\leq \widetilde{\sigma}_n(f),\quad M_n\leq l<M_{n+1},$$
by Remark \ref{control}, we have
\begin{equation}\label{sigma-ri}
\|(\sigma_l(f))_{l\geq1}\|_{\Lambda_{1,\infty}(\mathcal{N},\ell_{\infty})}\leq4\|(\widetilde{\sigma}_n(f))_{n\geq1}\|_{\Lambda_{1,\infty}(\mathcal{N},\ell_{\infty})}.
\end{equation}
The desired inequalities in Theorem \ref{CS}   follow from  Theorem \ref{sigma-tilde}.

\begin{theorem}\label{sigma-tilde}
Given a bounded sequence $m=(m_k)_{k\in\mathbb{N}}$,	let $\mathcal{N}=L_{\infty}(G_m)\bar{\otimes} \mathcal{M}$. 
	\begin{enumerate}[{\rm (i)}]
		\item If $p=1$, then there is a constant $c>0$ such that 	$$\|(\widetilde{\sigma}_n(f))_{n\geq1}\|_{\Lambda_{1,\infty}(\mathcal{N},\ell_{\infty})}\leq c \|f\|_{L_1(\mathcal{N})}.$$
		\item If $1<p\leq \infty$, then  there is a constant $c_p>0$ such that  $$\|(\widetilde{\sigma}_n(f))_{n\geq1}\|_{L_p(\mathcal{N},\ell_{\infty})}\leq  c_p\|f\|_{L_p(\mathcal{N})},$$
	where $c_p=\frac{cp^2}{(p-1)^2}$ for some positive constant $c$.
	\end{enumerate}
\end{theorem}

To prove Theorem \ref{sigma-tilde}, we establish the following two propositions. 

\begin{proposition}\label{Rbd}
Let $f\in L_1(\mathcal{N})$ be positve and $\lambda>0$. Suppose that $b_d$ is from \eqref{dec} associated with the given $f$ and $\lambda>0$. Then
	there exists a constant $c>0$ and  a projection $e_1\in \mathcal{P}(\mathcal{N})$ such that $e_1\leq q$,
	$$\|e_1\widetilde{\sigma}_n(b_d)e_1\|_{L_{\infty}(\mathcal{N})}\leq \lambda, \quad  \forall n\geq1$$
	and 
	$$\lambda\cdot\varphi({\bf 1}-e_1)\leq  c\|f\|_{L_1(\mathcal{N})}.$$
\end{proposition}

\begin{proposition}\label{Rboff}
Let $\lambda>0$ and positive $f\in L_1(\mathcal{N})$. Suppose that $b_{off}$ is from \eqref{dec} associated with respect to the given $f$ and $\lambda>0$. Then		there exists a constant $c>0$ and  a projection $e_2\in \mathcal{P}(\mathcal{N})$ such that $e_2\leq q$, 
	$$\|e_2\widetilde{\sigma}_n(b_{off})e_2\|_{L_{\infty}(\mathcal{N})}\leq \lambda, \quad  \forall n\geq1$$
	and 
	$$\lambda\cdot\varphi({\bf 1}-e_2)\leq c \|f\|_{L_1(\mathcal{N})}.$$
\end{proposition}

Denote by 
\begin{equation}\label{K+}
	\widetilde{K}_n=\sup_{M_{n}\leq l<M_{n+1}} |K_l|.
\end{equation} We also need the following new properties for the function $\widetilde{K}_n$, which are useful when we use 
noncommutative Calder\'{o}n-Zygmund decomposition later.

\begin{lemma}\label{K-E-t}
Let $a,n\in \mathbb{N}$.  If $n\geq a$, then, for each $s \in G_{m}$,  there exists   a constant $c>0$ such that
$$\int_{I_{a}(s) \backslash I_{a+1}(s)}\widetilde{K}_n(t,s) d \mu(t) \leq c\big(\delta^{a-n}+(n-a) \delta^{\frac{a-n}{2}}\big).
$$
\end{lemma}
\begin{proof} Let $l,n\in \mathbb{N}$ such that  $M_n\leq  l<M_{n+1}$. Recall that  $l_b$ and $l^{(b)}$ are as in \eqref{n-V-L} and \eqref{n^(k)-V-L-}. Note that
\begin{equation}\label{lKl}
lK_l=\sum_{b=0}^{n} \sum_{j=0}^{l_{b}-1} K_{l^{(b+1)}+j M_{b}, M_{b}},
\end{equation}
and $l_b\leq m_b-1$. Then
\begin{align}\label{K-I}
	\int_{I_{a}(s) \backslash I_{a+1}(s)} \widetilde{K}_n (t,s)d \mu(t) 
	&=\int_{I_a(s)\setminus I_{a+1}(s)} \sup_{M_n\leq l< M_{n+1}}|K_l (t,s)|d \mu(t)\nonumber\\
	&\leq \frac{1}{M_n}\int_{I_a(s)\setminus I_{a+1}(s)} \sup_{M_n\leq l< M_{n+1}}\sum_{b=0}^{n} \sum_{j=0}^{l_{b}-1} |K_{l^{(b+1)}+j M_{b}, M_{b}}(t,s)|d \mu(t)\nonumber\\
	&\leq \sum_{b=0}^{a} \sum_{j=0}^{m_{b}-2}\frac{1}{M_n}\int_{I_a(s)\setminus I_{a+1}(s)} \sup_{M_n\leq l< M_{n+1}} |K_{l^{(b+1)}+j M_{b}, M_{b}}(t,s)|d \mu(t)\\
	&\quad+\sum_{b=a+1}^{n} \sum_{j=0}^{m_{b}-2}\frac{1}{M_n}\int_{I_a(s)\setminus I_{a+1}(s)} \sup_{M_n\leq l< M_{n+1}} |K_{l^{(b+1)}+j M_{b}, M_{b}}(t,s)|d \mu(t)\nonumber\\
	&=:\mathrm{I}+\mathrm{II}.\nonumber
\end{align}
 By Lemma \ref{K-E-ab-1} and the assumption $M:=\sup_km_k<\infty$, we find that
$$\mathrm{I}\leq c\sum_{b=0}^{a} \sum_{j=0}^{m_{b}-2}\frac{1}{M_n}\frac{1}{M_{a}}\delta^{a-n}M_bM_n\leq c\delta^{a-n}.$$

It remains to estimate $\mathrm{II}$. By H\"{o}lder's inequality, we have
\begin{equation*}
	\mathrm{II}\leq \sum_{b=a+1}^{n} \sum_{j=0}^{m_{b}-2}\frac{1}{M_n}|I_a(s)\setminus I_{a+1}(s)|^{\frac{1}{2}} \left(\int_{I_a(s)\setminus I_{a+1}(s)} \sup_{M_n\leq l< M_{n+1}} |K_{l^{(b+1)}+j M_{b}, M_{b}}(t,s)|^2d \mu(t)\right)^{\frac{1}{2}}.
\end{equation*}
For $b,n\in \mathbb{N}$, set $$L_{b,n}:=\{l:l=\sum_{i=0}^n l_iM_i\quad \mbox{and}\quad i\in\{b+1,\dots,n\}\}.$$
 For each $M_n\leq l<M_{n+1}$, we have $l=\sum_{i=0}^n l_iM_i$, where $0\leq l_i\leq m_i-1$ for each $i$. Recall that $l^{(b+1)}=\sum_{i=b+1}^nl_iM_i$. 
Then we observe that, for each $t,s\in G_{m}$,
$$\sup_{M_n\leq l<M_{n+1}} |K_{l^{(b+1)}+j M_{b}, M_{b}}(t,s)|^2 =\sup_{l\in L_{b,n}} |K_{l^{(b+1)}+j M_{b}, M_{b}}(t,s)|^2.$$
Hence,  by Lemma \ref{K-E-ab-2}, we get 
\begin{equation*}
\begin{aligned}
\mathrm{II}&\leq \sum_{b=a+1}^{n} \sum_{j=0}^{m_{b}-2}\frac{1}{M_n}|I_a(s)\setminus I_{a+1}(s)|^{\frac{1}{2}} \Big(\int_{I_a(s)\setminus I_{a+1}(s)} \sup_{l\in L_{b,n}} |K_{l^{(b+1)}+j M_{b}, M_{b}}(t,s)|^2d \mu(t)\Big)^{\frac{1}{2}}\\
&\leq \sum_{b=a+1}^{n} \sum_{j=0}^{m_{b}-2}\frac{1}{M_n}|I_a(s)\setminus I_{a+1}(s)|^{\frac{1}{2}} \Big(\sum_{l\in L_{b,n}}\int_{I_a(s)\setminus I_{a+1}(s)}  |K_{l^{(b+1)}+j M_{b}, M_{b}}(t,s)|^2d \mu(t)\Big)^{\frac{1}{2}}\\
&\leq \sum_{b=a+1}^{n} \sum_{j=0}^{m_{b}-2}\frac{1}{M_n}\big(\frac{1}{M_{a}}\big)^{\frac{1}{2}} \Big(\sum_{l\in L_{b,n}} M_aM_bM_n\delta^{a-n} \Big)^{\frac{1}{2}}\\
&\leq c\sum_{b=a+1}^{n} \sum_{j=0}^{m_{b}-2}\frac{1}{M_n}\big(\frac{1}{M_{a}}\big)^{\frac{1}{2}}\cdot\Big(\frac{M_n}{M_b}M_aM_bM_n\delta^{a-n}\Big)^{\frac{1}{2}}\\
	&\leq c\sum_{b=a+1}^{n} \delta^{\frac{a-n}{2}}\leq c(n-a) \delta^{\frac{a-n}{2}},
\end{aligned}
\end{equation*}
where we used the fact $m_b\leq M:=\sup_km_k<\infty$.
The desired result follows from  the estimates of $\mathrm{I}$ and $\mathrm{II}$.
\end{proof}

\begin{lemma}\label{K^2-E-a} Let $a,n\in \mathbb{N}$.   If $n\geq a$, then, for each $s \in G_{m}$,  there exists a constant $c>0$ such that
$$\int_{I_{a}(s) \backslash I_{a+1}(s)}\widetilde{K}_n^2(t,s) d \mu(t) \leq c\big(M_a^{\frac{1}{2}}\delta^{a-n}+(n-a)\delta^{\frac{a-n}{2}}M_a^{\frac{1}{2}}\big)^2.
$$
\end{lemma}
\begin{proof} Similar to \eqref{K-I}, by \eqref{lKl}, we have
\begin{align*}
&\Big(\int_{I_{a}(s) \backslash I_{a+1}(s)}\widetilde{K}_n^2(t,s) d \mu(t)\Big)^{\frac{1}{2}}\\
	&\leq c\sum_{b=0}^{a} \sum_{j=0}^{m_{b}-2}\Big(\int_{I_a(s)\setminus I_{a+1}(s)}\sup_{M_n\leq l<M_{n+1}}
	\big[\frac{1}{M_n} |K_{l^{(b+1)}+j M_{b}, M_{b}}(t,s)|\big]^2d\mu(t)\Big)^{\frac{1}{2}}\\
	&\quad+c\sum_{b=a+1}^{n} \sum_{j=0}^{m_{b}-2}\Big( \int_{I_a(s)\setminus I_{a+1}(s)}\sup_{M_n\leq l<M_{n+1}}
	\big[\frac{1}{M_n}|K_{l^{(b+1)}+j M_{b}, M_{b}}(t,s)|\big]^2d\mu(t)\Big)^{\frac{1}{2}}\\
	&=:c(\mathrm{I'}+\mathrm{II'}).
\end{align*}
By Lemma \ref{K-E-ab-1}, we can obtain
\begin{align*}
\mathrm{I'}&\leq c\sum_{b=0}^{a} \sum_{j=0}^{m_{b}-2}	\left(\frac{1}{M_a}\big[\frac{1}{M_n}M_bM_n\delta^{a-n}\big]^2\right)^{\frac{1}{2}}=cM_a^{\frac{1}{2}}\delta^{a-n}\nonumber
\end{align*}
Similar to the proof of $\mathrm{II}$ in Lemma \ref{K-E-t}, by Lemma \ref{K-E-ab-2}, we have
\begin{align*}
\mathrm{II'} &= c\sum_{b=a+1}^{n} \sum_{j=0}^{m_{b}-2}\Big( \int_{I_a(s)\setminus I_{a+1}(s)}\sup_{M_n\leq l<M_{n+1}}
	\big[\frac{1}{M_n}|K_{l^{(b+1)}+j M_{b}, M_{b}}(t,s)|\big]^2d\mu(t)\Big)^{\frac{1}{2}}\\
	&\leq c\sum_{b=a+1}^{n} \sum_{j=0}^{m_{s}-2}\frac{1}{M_n}\cdot	\Big(\sum_{l\in L_{b,n}}\int_{I_a(s)\setminus I_{a+1}(s)}
	|K_{l^{(b+1)}+j M_{b}, M_{b}}(t,s)|^2d\mu(t)\Big)^{\frac{1}{2}}\\
	&\leq c\sum_{b=a+1}^{n} \sum_{j=0}^{m_{s}-2}\frac{1}{M_n}\Big(\frac{M_n}{M_b}M_aM_bM_n\delta^{a-n}\Big)^{\frac{1}{2}}\\
	&\leq c(n-a)\delta^{\frac{a-n}{2}}M_a^{\frac{1}{2}}.
\end{align*}
Then we get the desired inequality.
\end{proof}

\begin{lemma}\label{KI-E-lemma} Let  $n\in\mathbb{N}$. 
\begin{enumerate}[{\rm (i)}]
\item For each $\eta \in G_{m}$,  there exists an absolute constant $c>0$ such that
$$\sup_n\int_{G_m}\widetilde{K}_n(t,\eta)d\mu(t)\leq c.$$
\item As a consequence, we have
$$	\|\widetilde{\sigma}_n(f)\|_{L_{\infty}(\mathcal{N})}\leq c\|f\|_{L_{\infty}(\mathcal{N})}, \quad\forall n\geq1.$$
\end{enumerate}
\end{lemma}
\begin{proof} For each $n\in \mathbb{N}$ and $\eta\in G_m$,  we write
\begin{align*}
	\int_{G_m}\widetilde{K}_n(t,\eta)d\mu(t)&=\int_{G_m}\sup_{M_{n}\leq l<M_{n+1}}|K_l(t,\eta)|d\mu(t)\\
	 &\leq\int_{I_n(\eta)}\sup_{M_{n}\leq l<M_{n+1}}|K_l(t,\eta)|d\mu(t)\\
	 &\quad+\int_{G_m\setminus I_n(\eta)}\sup_{M_{n}\leq l<M_{n+1}}|K_l(t,\eta)|d\mu(t)\\
	&=:\mathrm{I}+\mathrm{II}.
\end{align*}
The estimate $\mathrm{II}\leq c$ is due to  Lemma \ref{K-E-GI}.

It remains to estimate $\mathrm{I}$. By \cite[Proposition 8]{Ga2001},  we have  
\begin{align*}
	D_l(t,\eta) &=
		\sum_{s=0}^{\infty} \psi_{l, s+1}(t) \bar{\psi}_{l, s+1}(\eta) D_{M_{s}}(t,\eta) \sum_{j=0}^{l_{s}-1} r_{s}^{l^{(s+1)}+j M_{s}}(t) \bar{r}_{s}^{l^{(s+1)}+j M_{s}}(\eta).
\end{align*}
Note that $l_s-1\leq m_s-2$ and $m_s\leq M:=\sup_km_k<\infty$.  Then, by Assumption \ref{gamma-nk} (iv), we have
\begin{align*}
	|D_l(t,\eta)| &\leq c\sum_{s=0}^{n}\left\|\psi_{l, s+1}\right\|_{\infty}^{2} M_{s}\sum_{j=0}^{m_{s}-2}\|r_{s}^{l^{(s+1)}+j M_{s}}\|_\infty^2\\ 
&\leq c\sum_{s=0}^{n} M_{s}\sum_{j=0}^{m_{s}-2}\frac{m_s}{\delta} \left\|\prod_{i=s+1}^n r_i^{l^{(i)}}\right\|_{\infty}^{2}\\ 
&\leq c\sum_{s=0}^{n} M_{s}m_s\frac{m_s}{\delta}\prod_{i=s+1}^n\frac{m_i}{\delta}\\
	&\leq c \sum_{s=0}^n M_{s} \delta^{s-n}\frac{M_{n}}{M_{s}}  \leq c M_n.
\end{align*}
Then $|K_l(t,\eta)|=|\frac{1}{l}\sum_{k=1}^lD_k(t,\eta)|\leq cM_n,$ which gives
\begin{align*}
	\mathrm{I}\leq c\int_{I_n(\eta)}M_nd\mu(t)=c.
\end{align*}
The  proof is complete.
\end{proof}

The proof of Proposition \ref{Rbd} relies on the following  lemma.

\begin{lemma}\label{Rnbd}
Let $k\geq1$ and $Q\in D(\mathcal{F}_k)$. Then, for each $n\geq1$, we have
\[
-q\sum_{n\geq k}\widetilde{\sigma}_n(|b_d^{k,Q}|)q\leq  q\widetilde{\sigma}_n(b_d^{k,Q})q\leq q\sum_{n\geq k}\widetilde{\sigma}_n(|b_d^{k,Q}|)q,
\]
where $b_d^{k,Q}=p_Q(f-f_Q)p_Q\chi_Q$. 
\end{lemma}
\begin{proof}
Fix $k\geq1$, $Q\in D(\mathcal{F}_k)$ and $n\geq1$. We  estimate 	$q(t)\widetilde{\sigma}_n(b_d^{k,Q})q(t)$ according to the cases $t\in Q$ and $t\in G_m\setminus Q$, respectively. We first  consider the case $t\in Q$. By Lemma \ref{Cuculescu}, $q\leq q_k$. Hence, $q(t)\leq q_k(t)=q_Q$ for $t\in Q$. Then, for each $i\geq1$, we can easily  check that 
\begin{align}\label{qEqbd}
q(t)\widetilde{\sigma}_n(b_d^{k,Q})(t)q(t)&=q(t)p_Q\widetilde{\sigma}_n(b_d^{k,Q})(t)p_Qq(t) \nonumber\\
		                                      &=q(t)q_Qp_Q\widetilde{\sigma}_n(b_d^{k,Q})(t)p_Qq_Qq(t)=0,
\end{align}
where we  used $p_Qq_Q=q_Qp_Q=0$ as in \eqref{pq}.
We assume from now on that $t\in G_m\setminus Q$. 
For each  $n$, by Assumption \ref{gamma-nk} (i,iv), it is easy to check that $\widetilde{K}_n$ is $\mathcal{F}_{n+1}$ measurable. Then, for $n\leq k-1$, we have
\begin{align}\label{sbd-nlessk-0}
	\widetilde{\sigma}_n(b_d^{k,Q})&=\int_{G_m}\widetilde{K}_np_Q(f-f_Q)p_Q\chi_Q\nonumber\\
	&=\int_{G_m}\mathbb{E}_k\big(\widetilde{K}_np_Q(f-f_Q)p_Q\chi_Q\big)\nonumber\\
	&=\int_{G_m}\widetilde{K}_np_Q\mathbb{E}_k(f-f_Q)p_Q\chi_Q=0.\nonumber
\end{align}
And for  the case $n\geq k$, it is clear that  
$$-\sum_{n\geq k}q\widetilde{\sigma}_n(|b_d^{k,Q}|)q\leq q\widetilde{\sigma}_n(b_d^{k,Q})q\leq \sum_{n\geq k}q\widetilde{\sigma}_n(|b_d^{k,Q}|)q,\quad t\in G_m\setminus Q.$$
The desired result follows. 
\end{proof}

\begin{proof}[Proof of Proposition \ref{Rbd}]
Denote
\[
\mathbb{B}_d=\sum_{k\geq1}\sum_{Q\in D(\mathcal{F}_k)}\sum_{n\geq k}\widetilde{\sigma}_n(|b_d^{k,Q}|)\chi_{G_m\setminus Q},
\]
and let
\[
e_1=\chi_{(0,\lambda)}(\mathbb{B}_d)\wedge q.
\]
Note that for each $n\geq1$
\[
	\widetilde{\sigma}_n(b_d)=({\bf 1}-q)	\widetilde{\sigma}_n(b_d)+q	\widetilde{\sigma}_n(b_d)({\bf 1}-q)+q	\widetilde{\sigma}_n(b_d)q.
\]
Hence, for each $ n\geq1$,
\[
e_1 \widetilde{\sigma}_n(b_{d}) e_1= e_1 q	\widetilde{\sigma}_n(b_{d})q e_1.
\]
In addition, since the operator $	\widetilde{\sigma}_n$ is linear, by Lemma \ref{Rnbd} and \eqref{bd}, it follows that 
\begin{align*}
-q\mathbb{B}_dq&\leq q	\widetilde{\sigma}_n(b_d)q=\sum_{k\geq1}\sum_{Q\in D(\mathcal{F}_k)}q	\widetilde{\sigma}_n(b_d^{k,Q})q\leq q\mathbb{B}_dq, \quad \forall n\geq 1.
\end{align*}
Then, for each $n\geq1$,
\begin{equation}\label{Rnbe1}
-\lambda \leq -e_1\mathbb{B}_de_1\leq e_1 \widetilde{\sigma}_n(b_{d})e_1 \leq e_1\mathbb{B}_de_1\leq \lambda.
\end{equation}
On the other hand, by the Chebyshev inequality, 
\begin{align}\label{Bd-w11}
	\lambda\varphi(\chi_{(\lambda,\infty)}(\mathbb{B}_d))
	&\leq \|\mathbb{B}_d\|_{L_1(\mathcal{N})}\\
	&=\sum_{k\geq 1}\sum_{Q\in D(\mathcal{F}_k)}\sum_{n\geq k}\|\widetilde{\sigma}_n(|b_d^{k,Q}|)\chi_{G_m\setminus Q}\|_{L_1(\mathcal{N})}.\nonumber
\end{align}
On the other hand, by \eqref{qEqbd},  we have
\begin{equation}\label{qbdq}
\begin{aligned}
	\|\widetilde{\sigma}_n(|b_d^{k,Q}|)\chi_{G_m\setminus Q}\|_{L_1(\mathcal{N})}&=\int_{G_m\setminus Q} \big\|\int_Q\widetilde{K}_n (t,s)\cdot |b_d^{k,Q}| d \mu(s)\big\|_{L_1(\mathcal{M})}d \mu(t)\\
	&=\int_Q \|b_d^{k,Q}\|_{L_1(\mathcal{M})} \int_{G_m\setminus Q} \widetilde{K}_n (t,s)d \mu(t) d\mu(s).
\end{aligned}
\end{equation}
By Lemma \ref{K-E-t}, we get
\begin{align}\label{qbdq-K}
	\int_{G_m\setminus Q} \widetilde{K}_n (t,s)d \mu(t) &=\sum_{a=0}^{k-1}\int_{I_a(s)\setminus I_{a+1}(s)} \widetilde{K}_n (t,s)d \mu(t)\nonumber\\
	&\leq c\sum_{a=0}^{k-1}\big(\delta^{a-n}+(n-a) \delta^{\frac{a-n}{2}}\big)\\
    &\leq c\big(\delta^{k-n}+(n-k) \delta^{\frac{k-n}{2}}\big)\nonumber.
\end{align}
Then, by \eqref{Bd-w11}, \eqref{qbdq} and \eqref{qbdq-K} and Theorem \ref{NewCZ} (ii), we have

\begin{align}\label{Bd-w11-last}
	\lambda\varphi(\chi_{(\lambda,\infty)}(\mathbb{B}_d))
	&\leq c\sum_{k\geq 1}\sum_{Q\in D(\mathcal{F}_k)}c\sum_{n\geq k}\big(\delta^{k-n}+(n-k)\delta^{\frac{k-n}{2}}\big)\|b_d^{k,Q}|\|_{L_1(\mathcal{N})}\nonumber\\
	&\leq c\sum_{k\geq 1}\sum_{Q\in D(\mathcal{F}_k)}\|b_d^{k,Q}|\|_{L_1(\mathcal{N})}\\
	&\leq c\|f\|_{L_1(\mathcal{N})}.\nonumber
\end{align}
where the constant is independent of  $f$. Then, by Lemma  \ref{Cuculescu} (iv),
\begin{align*}
\lambda\varphi(1-e_1)&\leq\lambda\varphi(\chi_{(\lambda,\infty)}(\mathbb{B}_d))+\lambda\varphi(1-q)\\
&\leq c\|f\|_{L_1(\mathcal{N})},
\end{align*}
which completes our proof of Proposition \ref{Rbd}.
\end{proof}

The proof of Proposition \ref{Rboff} relies on the next lemma which is a cousin of Lemma \ref{Rnbd}. However, in the following setting, we take the absolute value outside the operator $\widetilde{\sigma}_n$. 
\begin{lemma}\label{Rnboff}
	For each $k\geq1$ and $Q\in D(\mathcal{F}_k)$. Then, for each $n\geq1$, we have
	$$-qA_{k,Q}q\leq q	\widetilde{\sigma}_n(b_{off}^{k,Q})q\leq qA_{k,Q}q,$$
	where   $b_{off}^{k,Q}=p_Q(f-f_Q)q_Q\chi_Q+q_Q(f-f_Q)p_Q\chi_Q,$  and 
$$A_{k,Q}=\sum_{n\geq k}|\widetilde{\sigma}_n(b_{off}^{k,Q})(t)|\chi_{G_m\setminus Q}.$$
\end{lemma}

\begin{proof} Considering $t\in Q$, we have 
\begin{align*}
q(t)\widetilde{\sigma}_n\Big(p_Q(f-f_Q)q_Q\chi_Q\Big)(t)q(t)&=q(t)p_Q\widetilde{\sigma}_n\Big((f-f_Q)q_Q\chi_Q\Big)(t)q(t)\\
&=q(t)q_Qp_Q\mathbb{E}_{i}^j\Big((f-f_Q)q_Q\chi_Q\Big)(t)q(t)=0,
\end{align*}
where we used $q_Qp_Q=0$ as in \eqref{pq}. Similarly, using $p_Qq_Q=0$, we obtain
\[
q(t)\widetilde{\sigma}_n\Big(q_Q(f-f_Q)p_Q\chi_Q\Big)(t)q(t)=0,\quad t\in Q.
\]
Hence
\begin{align*}
q(t)\widetilde{\sigma}_n(b_{off}^{k,Q})(t)q(t)=0,\quad t\in Q.
\end{align*}	
	
It remains to estimate $q(t)\widetilde{\sigma}_n(b_{off}^{k,Q})(t)q(t)$ for $t\in G_m\setminus Q$. Similar to Lemma \ref{Rnbd}, we have

$$-q(t)\sum_{n\geq k}|\widetilde{\sigma}_n(b_{off}^{k,Q})(t)|q(t)\leq q(t)\widetilde{\sigma}_n(b_{off}^{k,Q})(t)q(t)\leq q(t)\sum_{n\geq k}|\widetilde{\sigma}_n(b_{off}^{k,Q})(t)|q(t).$$
The proof is complete.
\end{proof}

The following lemma is the key estimate in our proof of Proposition \ref{Rboff}.

\begin{lemma}\label{AkQ}
Suppose that $f$ is a positive element in $L_{1}(\N)$ and $\lambda>0$. For each $k\geq1$ and $Q\in D( \mathcal{F}_k)$, we have
	$$\|A_{k,Q}\|_{L_1(\mathcal{N})}\leq c \lambda^{1/2}\varphi(p_Q\chi_Q)^{1/2}\varphi(p_Qfp_Q\chi_Q)^{1/2},$$
	where $A_{k,Q}$ is as in Lemma \ref{Rnboff}. 
\end{lemma}
\begin{proof}
Fix $k$ and $Q\in D( \mathcal{F}_k)$. 	By Lemma \ref{Cuculescu} (ii), for every $k\geq1$, 
$$p_kf_kq_k=p_kq_{k-1}f_kq_{k-1}q_k=p_kq_kq_{k-1}f_kq_{k-1}=0.$$ Hence $p_Qf_Qq_Q=q_Qf_Qp_Q=0$. Then we have 
\begin{align*}
\|A_{k,Q}\|_{L_1(\mathcal{N})}&\leq \sum_{n\geq k}  \|\widetilde{\sigma}_n(b_{off}^{k,Q})\chi_{G_m\setminus Q}\|_{L_1(\mathcal{N})}\\
&\leq 2 \sum_{n\geq k}  \|\widetilde{\sigma}_n(p_Qfq_Q)\chi_{G_m\setminus Q}\|_{L_1(\mathcal{N})}\\
&= 2\sum_{n\geq k} \int_{G_m\setminus Q} \big\|\int_Q\widetilde{K}_n (t,s)\cdot p_Qf(s)q_Q d \mu(s)\big\|_{L_1(\mathcal{M})}d \mu(t)\\
&=2\sum_{n\geq k} \int_{G_m\setminus Q} \big\|A_n(t)^{\frac{1}{2}}\cdot u_Q\cdot B^{\frac{1}{2}}\big\|_{L_1(\mathcal{M})}d \mu(t),
\end{align*}
where  the last equality is due to Lemma \ref{AcB} with
\[
A_n(t)=\int_Q\widetilde{K}_n^2 (t,s)\cdot p_Qf(s)p_Qd \mu(s),\quad B=\int_Qq_Qf(s)q_Q d \mu(s),\quad \|u_Q\|_{L_{\infty}(\M)}\leq 1.
\]
Then H\"{o}lder's inequality gives us 
\begin{align*}
\|A_{k,Q}\|_{L_1(\mathcal{N})}\leq 2\sum_{n\geq k} \int_{G_m\setminus Q} \big\|A_n(t)^{\frac{1}{2}}\big\|_{L_1(\mathcal{M})}\cdot \|u_Q\|_{L_\infty(\mathcal{M})}\cdot \big\|B^{\frac{1}{2}}\big\|_{L_\infty(\mathcal{M})}d \mu(t).
\end{align*}
According to Lemma \ref{Cuculescu} (iii), it follows that
\begin{equation}\label{pfq2}
\begin{aligned}
\big\|B^{\frac{1}{2}}\big\|_{L_\infty(\mathcal{M})}&= \Big\|\Big(\int_Qq_Qf(s)q_Q d\mu(s)\Big)^{1/2}\Big\|_{L_\infty(\mathcal{M})}\\
&=\Big\||Q|\cdot q_Qf_Qq_Q\Big\|_{L_\infty(\mathcal{M})}^{1/2}\leq |Q|^{1/2}\lambda^{1/2}.
\end{aligned}
\end{equation}
Using H\"{o}lder's inequality, we obtain 
\begin{align*}
\|A_n(t)^{1/2}\|_{L_1(\mathcal{M})}=\|p_QA_n(t)\|_{L_{1/2}(\mathcal{M})}^{1/2}\leq \|p_Q\|_{L_1(\mathcal{M})}^{1/2} \|A_n(t)\|_{L_1(\mathcal{M})}^{1/2}.
\end{align*}
Then by the above estimates, we can obtain
\begin{align*}
\|A_{k,Q}\|_{L_1(\mathcal{N})}&\leq 2\sum_{n\geq k} \int_{G_m\setminus Q} \big\|A_n(t)\big\|_{L_1(\mathcal{M})}^{\frac{1}{2}}\cdot\lambda^{1/2}|Q|^{1/2}\|p_Q\|_{L_1(\mathcal{M})}^{1/2} d \mu(t)\\
&\leq 2\lambda^{1/2}\varphi(p_Q\chi_Q)^{1/2}\sum_{n\geq k} \int_{G_m\setminus Q} \big\|A_n(t)\big\|_{L_1(\mathcal{M})}^{\frac{1}{2}}d \mu(t).
\end{align*}
By H\"{o}lder's inequality and Fubini theorem,  for each $s\in Q$,  we get
\begin{align*}
\int_{G_m\setminus Q}& \|A_n(t)\|_{L_1(\M)}^{1/2}d\mu(t)\\
&=\sum_{a=0}^{k-1}\int_{I_a(s)\setminus I_{a+1}(s)} \|A_n(t)\|_{L_1(\M)}^{1/2}d\mu(t) \\
&\leq \sum_{a=0}^{k-1}|I_a(s)\setminus I_{a+1}(s)|^{1/2}\Big(\int_{I_a(s)\setminus I_{a+1}(s)}\|A_n(t)\|_{L_1(\M)}d\mu(t)\Big)^{1/2}\\
&=\sum_{a=0}^{k-1}M_a^{-1/2}\Big(\int_{I_a(s)\setminus I_{a+1}(s)}\|\int_Q\widetilde{K}_n^2 (t,s)\cdot p_Qf(s)p_Qd \mu(s)\|_{L_1(\M)}d\mu(t)\Big)^{1/2}\\
&=\sum_{a=0}^{k-1}M_a^{-1/2}\Big(\int_Q\|p_Qf(s)p_Q\|_{L_1(\M)}\int_{I_a(s)\setminus I_{a+1}(s)}\widetilde{K}_n^2 (t,s)d\mu(t)d \mu(s)\Big)^{1/2}.
\end{align*}
By Lemma \ref{K^2-E-a},
we conclude from the above argument that
\begin{align*}
& \|A_{k,Q} \|_{L_1(\mathcal{N})} \\
&\leq c\sum_{n\geq k}\sum_{a=0}^{k-1}M_a^{-1/2}\big(M_a^{\frac{1}{2}}\delta^{a-n}+(n-a)\delta^{\frac{a-n}{2}}M_a^{\frac{1}{2}}\big)\lambda^{1/2}\varphi(p_Q\chi_Q)^{1/2}\cdot\varphi(p_Qfp_Q\chi_Q)^{1/2}\\
&\leq c\sum_{n\geq k}\sum_{a=0}^{k-1}\big(\delta^{a-n}+(n-a)\delta^{\frac{a-n}{2}}\big)\lambda^{1/2}\varphi(p_Q\chi_Q)^{1/2}\cdot\varphi(p_Qfp_Q\chi_Q)^{1/2}\\
&=\leq c\sum_{n\geq k}\big(\delta^{k-n}+(n-k)\delta^{\frac{a-k}{2}}\big)\lambda^{1/2}\varphi(p_Q\chi_Q)^{1/2}\cdot\varphi(p_Qfp_Q\chi_Q)^{1/2}\\
&\leq c\lambda^{1/2}\varphi(p_Q\chi_Q)^{1/2}\cdot\varphi(p_Qfp_Q\chi_Q)^{1/2}.
\end{align*}
The proof is complete.
\end{proof}

We  now  provide the proof of Proposition \ref{Rboff}.
\begin{proof}[Proof of Proposition \ref{Rboff}]
	Denote
	$$\mathbb{B}_{off}=\sum_{k\geq1}\sum_{Q\in D(\mathcal{F}_k)}A_{k,Q},$$
	and let 
	$$e_2=\chi_{(0,\lambda)}(\mathbb{B}_{off})\wedge q.$$
 From the linearity of   $\widetilde{\sigma}_n$, Lemma \ref{Rnboff} and \eqref{boff}, it follows that 
	\begin{equation}\label{Rnboffe2}
		-q\mathbb{B}_{off}q\leq q\widetilde{\sigma}_n(b_{off})q=\sum_{k\geq1}\sum_{Q\in D(\mathcal{F}_k)}q\widetilde{\sigma}_n(b_{off}^{k,Q})q\leq q\mathbb{B}_{off}q,\quad \forall n\geq1.
	\end{equation}
Consequently, for each $n\geq1$,
	\begin{equation*}\label{Rnbe2}
		-\lambda \leq -e_2 \mathbb{B}_{off} e_2\leq  e_2 q\widetilde{\sigma}_n(b_{off})q e_2\leq e_2\mathbb{B}_{off}e_2 \leq \lambda,
	\end{equation*}
	By the Chebyshev inequality and  Lemma \ref{AkQ}, we have 
	\begin{align*}
	\lambda \varphi(\chi_{(\lambda,\infty)}(\mathbb{B}_{off}))
		&\leq \|\mathbb{B}_{off}\|_{L_1(\mathcal{N})}\\
		&\leq \sum_{k\geq1}\sum_{Q\in D(\mathcal{F}_k)}\|A_{k,Q}\|_{L_1(\mathcal{N})}\\
		&\leq c\sum_{k\geq1}\sum_{Q\in D(\mathcal{F}_k)}   \lambda^{1/2}\varphi(p_Q\chi_Q)^{1/2}\cdot \varphi(p_Qfp_Q\chi_Q)^{1/2}\\
		&\leq c\Big(\sum_{k\geq1}\sum_{Q\in D(\mathcal{F}_k)}  \lambda\varphi(p_Q\chi_Q)\Big)^{1/2}\cdot \Big(\sum_{k\geq1}\sum_{Q\in D(\mathcal{F}_k)}  \varphi(p_Qfp_Q\chi_Q)\Big)^{1/2},
	\end{align*}
	where we  used H\"{o}lder's inequality in the last inequality. 
	Applying Lemma \ref{Cuculescu} (iv), 
	$$\sum_{k\geq1}\sum_{Q\in D(\mathcal{F}_k)}  \lambda\varphi(p_Q\chi_Q)=\lambda \sum_{k\geq1}\varphi(p_k)=\lambda \varphi({\bf 1}-q)\leq\varphi(({\bf 1}-q)f)\leq  \|f\|_{L_1(\mathcal{N})}.$$
	In addition, 
	$$\sum_{k\geq1}\sum_{Q\in D(\mathcal{F}_k)}  \varphi(p_Qfp_Q\chi_Q)=\sum_{k\geq1}\varphi(fp_k)=\varphi(({\bf 1}-q)f)\leq \|f\|_{L_1(\mathcal{N})}.$$
	Therefore, by the above inequalities and Lemma \ref{Cuculescu} (iv), 
	$$\lambda \varphi({\bf 1}-e_2)\leq 	\lambda \varphi(\chi_{(\lambda,\infty)}(\mathbb{B}_{off}))+\varphi(1-q)\leq c\|f\|_{L_1(\mathcal{N})},$$
which completes our proof of Proposition \ref{Rboff}.
\end{proof}

With  Propositions  \ref{Rbd} and \ref{Rboff} at hand, we now provide a proof of Theorem \ref{sigma-tilde} in detail.

\begin{proof}[Proof of Theorem \ref{sigma-tilde}]
It follows from Lemma \ref{KI-E-lemma} that
\[
\|(\widetilde{\sigma}_n(f))_{n\geq1}\|_{L_{\infty}(\mathcal{N},\ell_{\infty})}=\sup_n\|\widetilde{\sigma}_n(f)\|_{L_{\infty}(\mathcal{N})}\leq c\|f\|_{L_{\infty}(\mathcal{N})}.
\]
Note that   $\widetilde{\sigma}_n$ is   positive for each $n\geq1$. Then, once item (i) is proved,  we can apply the noncommutative Marcinkiewicz interpolation in \cite[Theorem 3.1]{JX2007} to get item (ii).

For item (i), without loss of generality, we assume that $f\geq0$. According to the definition of  $\|\cdot\|_{\Lambda_{1,\infty}(\mathcal{N},\ell_{\infty})}$,  it suffices to find a projection $e\in \mathcal{P}(\mathcal{N})$ such that, for each fixed $\lambda>0$, 
	\begin{equation*}\label{Fne}
		\sup_n\|e\widetilde{\sigma}_n(f)e\|_{L_{\infty}(\mathcal{N})}\leq c\lambda
	\end{equation*}
	and 
	\begin{equation}\label{traceF}
		\lambda\cdot\varphi({\bf 1}-e)\leq c\|f\|_{L_1(\mathcal{N})}.
	\end{equation} 
	For the fixed $\lambda>0$, applying Theorem \ref{NewCZ}, we have 
	\begin{equation*}\label{dec-f}
		f=g+b_d+b_{off}.
	\end{equation*}
Set
\[
e=e_1\wedge e_2,
\]
where $e_1$ and $e_2$ are from Proposition  \ref{Rbd} and Proposition  \ref{Rboff}, respectively. Then, combining Proposition  \ref{Rbd}, Proposition \ref{Rboff} and Theorem \ref{NewCZ} (i), we can obtain 
\begin{align*}
\sup_{n\geq 1}&\|e\widetilde{\sigma}_n(f)e\|_{L_{\infty}(\mathcal{N})}\\
&\leq \sup_{n\geq 1}\|e\widetilde{\sigma}_n(g)e\|_{L_{\infty}(\mathcal{N})}+\sup_{n\geq 1}\|e\widetilde{\sigma}_n(b_d)e\|_{L_{\infty}(\mathcal{N})}+\sup_{n\geq 1}\|e\widetilde{\sigma}_n(b_{off})e\|_{L_{\infty}(\mathcal{N})} \leq c\lambda
\end{align*}
and
\[
\lambda\varphi({\bf 1}-e)\leq c\|f\|_{L_1(\mathcal{N})}.
\]
The proof is complete. 
\end{proof}

%
%

\subsection{Proof of Theorem \ref{Nc-W}}\label{sec-pf-Nc}
This subsection aims to prove Theorem   \ref{Nc-W}.
Recall that $\widehat{\mathcal{R}}\coloneqq\gamma(\mathcal{R})$, where $\gamma$ is as in \eqref{transference mapping}. The next lemma is motivated by \cite{CXY2013}.

\begin{lemma}\label{trans-maximal}
	Let $1\leq p<\infty$. Let $\mathcal{A}\subset \mathcal{B}$ be two von Neumann algebras. Assume that there exists a conditional expectation $\mathcal{E}:\mathcal{B}\to\mathcal{A}$. Suppose that $x=(x_{n})_{n\in \Na}$ is a sequence of self-adjoint elements in $L_{p,\infty}(\mathcal{A})$, and $f=(f_n)_{n\in \Na}$ is a sequence of positive elements in $L_{p,\infty}(\mathcal{B})$ such that $-f_n\leq x_n\leq f_n$ for each $n\in \Na$. Then
	\begin{equation*}
	 \|x\|_{\Lambda_{p,\infty}(\mathcal{A},\ell_{\infty})}\leq 2^{2+\frac{1}{p}}\|f\|_{\Lambda_{p,\infty}(\mathcal{B},\ell_{\infty})}.
	\end{equation*}
\end{lemma}

\begin{proof}
For arbitrary $\lambda>0$,   take a projection $e\in \mathcal{P}(\mathcal{B})$ such that $$\sup_{n\geq1}\|ef_ne\|_{L_\infty({\mathcal{B}})}\leq \lambda.$$
By the contraction of conditional expectation $\Ex:\mathcal{B}\to \mathcal{A}$, we get that, for each $n\in \Na$,
	\begin{align*}
		\left\|\Ex(e)f_n^{1/2}\right\|_{L_\infty({\mathcal{B}})}&=\left\|\Ex(ef_n^{1/2})\right\|_{L_\infty({\mathcal{B}})}\\
		&\leq  \|ef_n^{1/2}\|_{L_\infty({\mathcal{B}})}=\|ef_ne\|_{L_\infty({\mathcal{B}})}^{1/2}\leq \lambda^{1/2}.
	\end{align*}
	Let $a=\Ex(e)$, and let $\xi=\chi_{[1/2,1]}(a)$ which is a projection in $\mathcal{A}$ . The functional calculus gives us $\1-\xi=\chi_{[1/2,1]}(\1-a)$. Hence, $\1-\xi\leq 2 (\1-a)$, and 
	\begin{equation}\label{pe2}
		\varphi(\1-\xi)\leq 2\varphi(\1-a)=2\varphi(\1-e).
	\end{equation}
	Moreover, taking $g(t)=\frac{1}{t}\chi_{[1/2,1]}(t)$, we have $\xi=\xi g(a)a$.  Note that  $\|\xi g(a)\|_{L_{\infty}(\mathcal{B})}\leq 2$. For any $n\geq1$, we obtain
	\begin{align*}
		\|\xi x_n\xi\|_{L_\infty({\mathcal{A}})}\leq \|\xi f_n\xi\|_{L_\infty({\mathcal{B}})}
		=\|\xi g(a) a f_n^{1/2}\|_{L_\infty({\mathcal{B}})}^2\leq 4  \|\Ex(e) f_n ^{1/2}\|_{L_\infty({\mathcal{B}})}^2\leq 4\lambda,
	\end{align*}
which, together with \eqref{pe2} and the definition of quasi-norm $\|\cdot\|_{\Lambda_{p,\infty}}$, implies the desired inequality.
\end{proof}

The next lemma is basic, however, we still include its detail proof for the reader's convenience.
\begin{lemma}\label{equvalence} Let $1\leq p<\infty$. For any $x=(x_{n})_{n\in \Na}\in \Lambda_{p,\infty}(\Ra,\ell_{\infty})$, we  have 
$$\|\gamma(x)\|_{\Lambda_{p,\infty}(\widehat{\Ra},\ell_{\infty})}= \|x\|_{\Lambda_{p,\infty}(\Ra,\ell_{\infty})}.$$
\end{lemma}
\begin{proof}
 For every $\lambda>0$, take a projection $e\in \mathcal{P}(\Ra)$ such that
	\begin{equation*}
		\sup_{n\geq1}\|ex_{n}e\|_{\Ra}\leq \lambda.
	\end{equation*}
	Note that $\gamma$ is a $*$-homomorphism (see Proposition \ref{transference 1-2}). Then
	\begin{equation*}
		\gamma(e)^{*}=\gamma(e),\quad
		\gamma(e)^{2}=\gamma(e^{2})=\gamma(e),
	\end{equation*}
	which implies that $\gamma(e)$ is also a projection in $\widehat{\Ra}=\gamma(\mathcal{R})$. 
	Moreover, 	it is easy to verify   via Theorem \ref{transference argument} that, for any $n\in \Na$,
	\begin{equation}\label{gamma-e}
			\left\|\gamma(e)\gamma(x_{n})\gamma(e)\right\|_{\widehat{\Ra}}=\left\|\gamma\left(ex_{n}e\right)\right\|_{L_{\infty}(G_{2m})\bar{\otimes }\mathcal{R}}=\|ex_{n}e\|_{\mathcal{R}}
			\leq \lambda.
	\end{equation}
Since $\gamma$ is trace preserving, it follows that
\[
\|\gamma(x)\|_{\Lambda_{p,\infty}(L_{\infty}(\widehat{\Ra},\ell_{\infty})}\leq \|x\|_{\Lambda_{p,\infty}(\Ra,\ell_{\infty})}.
\]

Conversely, take a projection $p\in \widehat{\Ra}$ such that 	\begin{equation*}
		\sup_{n\geq1}\|p\gamma(x_{n})p\|_{\Ra}\leq \lambda.
	\end{equation*}
Then, there exists $e\in \Ra$ such that $p=\gamma(e)$. Thus, the same to \eqref{gamma-e}, we have $\sup_n\|ex_ne\|_{\mathcal{R}}\leq \lambda$. 
Again, since $\gamma$ is trace preserving, it follows that
\[
 \|x\|_{\Lambda_{p,\infty}(\Ra,\ell_{\infty})}\leq \|\gamma(x)\|_{\Lambda_{p,\infty}(L_{\infty}(\widehat{\Ra},\ell_{\infty})}.
\]
The proof is complete.
\end{proof}

Now we are ready to prove Theorem \ref{Nc-W} in detail. 
\begin{proof}[Proof of Theorem \ref{Nc-W}]
Let $x\in L_{1}(\mathcal{R})$. Without loss of generality, we assume that $x$ is positive. Since $\gamma$ is a   $*$-homomorphism (see Proposition \ref{transference 1-2}), $\gamma(x)\geq0$.   Note that
\begin{align*}
\gamma (\sigma^{\mathcal{R}}_n(x))(t)&=\sigma_{n}(\gamma(x))(t), \quad t\in G_{2m},
\end{align*}
where $\sigma_{n}$ is as in \eqref{CM}.
Then, by Lemma \ref{equvalence}, we have
\begin{align}\label{1-equality}
\left\|(\sigma^{\mathcal{R}}_n(x))_{n\geq1}\right\|_{\Lambda_{1,\infty}(\mathcal{R},\ell_{\infty})}= \left\|(\gamma[\sigma^{\mathcal{R}}_n(x)])_{n\geq1}\right\|_{\Lambda_{1,\infty}(\widehat{\mathcal{R}},\ell_{\infty})}=\left\|(\sigma_n(\gamma (x)))_{n\geq1}\right\|_{\Lambda_{1,\infty}(\widehat{\mathcal{R}},\ell_{\infty})}.
\end{align}
Combining \eqref{sigma-ri}, 
 Lemma \ref{trans-maximal}, Theorem \ref{sigma-tilde} and Theorem \ref{transference argument}, we arrive at
\begin{align*}
\|(  \sigma_{n}(\gamma(x)))_{n\geq1}\|_{\Lambda_{1,\infty}(\widehat{\mathcal{R}},\ell_{\infty})}&\leq 4\|( \widetilde{\sigma}_n(\gamma(x))_{n\geq1}\|_{\Lambda_{1,\infty}(\widehat{\mathcal{R}},\ell_{\infty})}\\
&\leq c\|( \widetilde{\sigma}_n(\gamma(x))_{n\geq1}\|_{\Lambda_{1,\infty}(L_{\infty}(G_{2m})\bar{\otimes}\mathcal{R},\ell_{\infty})}\\
&\leq c\|\gamma(x)\|_{L_1(L_{\infty}(G_{2m})\bar{\otimes}\mathcal{R})}\\
&=c\|x\|_{L_1(\mathcal{R})},
\end{align*}
where $\widetilde{\sigma}_n$ is as in \eqref{s-tilde}.  Then, combining \eqref{1-equality}, we get (i) of Theorem \ref{Nc-W}.

Similarly, we can show the strong type $(p,p)$ inequality, and the proof is complete. 
\end{proof}

\section{Proofs of Theorems \ref{main-asy-op} and \ref{main-asy}}\label{sec-pf-main1}

In Section \ref{sec-Hardy}, we introduce the noncommutative Hardy space and state its atomic decomposition. In Section \ref{sec-Sun23}, we study a generalization of noncommutative Sunouchi operator, and then use it to prove Theorem \ref{main-asy-op} and Theorem \ref{main-asy} in Section \ref{sec-pf asy op}.

\subsection{Noncommutative martingale Hardy spaces and atomic decomposition}\label{sec-Hardy}
Consider a noncommutative martingale $f=(f_n)_{n\geq1}$ with respect to the regular filtration $(\N_n)_{n\geq1}$ as in Example \ref{one-d}. 
The related column square functions   are defined by
$$
S_{c,n}(f)=\Big(\sum_{k=1}^n|df_k|^2\Big)^{1/2}, \quad S_{c}(f)=\Big(\sum_{k\geq1}|df_k|^2\Big)^{1/2}.$$
The row square functions $S_{r,n}(f)$ and $S_r(f)$  can be defined by taking adjoint.
For $1\leq p\leq 2$, define the noncommutative martingale spaces as follows
$$H_p^c(\N)=\{f\in L_p(\N):\|f\|_{H_p^c(\N)}=\|S_c(f)\|_{L_p(\N)}<\infty\}$$
and 
$$H_p^r(\N)=\{f\in L_p(\N):\|f\|_{H_p^r(\N)}=\|S_r(f)\|_{L_p(\N)}<\infty\}.$$

\begin{definition}\label{def-atom}
	An element $a\in L_2(\N)$ is said to be a $(1,2)_c$-simple  atom with respect to the filtration $(\N_n)_{n\geq1}$ if there exist  $k\geq1$, $Q\in D(\mathcal{F}_k)$, and $e_Q\in \mathcal{P}(\mathcal{M})$  such that 
	\begin{enumerate}[{\rm (i)}]
		\item $a=ae$ with $e=e_Q\chi_Q$;
		\item $\mathbb{E}_k(a)=0$;
		\item $\|a\|_{L_2(\N)}\leq \varphi(e)^{-1/2}=\tau(e_Q)^{-1/2}|Q|^{-1/2}$.
	\end{enumerate}
	Simple row atoms are defined to satisfy $a=ea$ instead.
\end{definition}

\begin{remark}\label{use-rem} Assume that $a$ is a $(1,2)_c$-simple   atom associated with some  projection $e=e_Q\chi_Q\in \N_k$ with $Q\in D(\mathcal{F}_{k})$, $k\geq1$.
	\begin{enumerate}[{\rm (i)}]
		\item From the above definition, we have the following fact
		\begin{equation}\label{aL1}
			\|a\|_{L_1(\N)}=\|ae\|_{L_1(\N)}\leq \|a\|_{L_2(\N)}\|e\|_{L_2(\N)}\leq 1.
		\end{equation}
		\item Every $(1,2)_c$-simple atom is an $\mathcal{M}_c$-atom introduced by Mei in \cite[p.\,26]{Me2007}. Indeed, by the Cauchy-Schwarz  inequality, we have 
		\[
		\tau\Big[\Big(\int_Q|a|^2dt\Big)^{1/2}\Big]\leq [\tau(e_Q)]^{1/2} \Big[\tau\Big(\int_Q|a|^2dt\Big)\Big]^{1/2}\leq |Q|^{-1/2}.
		\]
		\item Every $(1,2)_c$-simple atom is also a $(1,2)_c$-atom in the sense of \cite[Definition 2.1]{BCPY2010}.
	\end{enumerate}
	
\end{remark}

\begin{definition}
	Define $H_1^{c, \mathrm{sat}}(\mathcal{N})$ as the Banach space of all $f\in L_1(\N)$ which admits a decomposition
	\[
	f=\sum\limits_{k\geq 1}\lambda_{k}a_{k}
	\]
	with, for each $k\geq 1$, $a_k$ is a   $(1,2)_c$-simple atom or an element in $L_1(\N_1)$ of norm $\leq 1$, and $\sum_{k\geq 1} |\lambda_k|<\infty$. We equip this space with the norm
	\[
	\|f\|_{H^{c,\mathrm{sat}}_{1}}\coloneqq\inf\sum\limits_{k\geq 1}|\lambda_{k}|,
	\]
	where the infimum is taken over all decomposition of $f$ described as above. The row space $H_1^{r, \mathrm{sat}}(\mathcal{N})$ can be defined by a similar way.
\end{definition}


Atomic decomposition for noncommutative martingale Hardy spaces  has been studied extensively by many authors; see for instance \cite{BCPY2010, CRX2023, HM2012}. Since the filtration $(\N_n)_{n\geq1}$ is regular (see Example \ref{one-d}), it follows that $H_1^c(\N)$ is just the conditioned Hardy space $\mathrm{h}_1^c(\N)$ (see \cite{CRX2023}). Consequently, 
according to \cite[Theorem 2.4]{BCPY2010}, $H_1^c(\N)$ admits atomic decomposition   for the so-called $(1,2)_c$-atoms. According to Remark \ref{use-rem} (iii), the following inclusion holds
\[
H_1^{c,\mathrm{sat}}(\N)\subseteq  H_1^c(\N).
\]
On the other hand, the duality argument used in \cite[Theorem 2.4]{BCPY2010} (or \cite{HM2012}) works well for the reverse inclusion $H_1^c(\N)\subseteq H_1^{c,\mathrm{sat}}(\N)$. Indeed, we can show
\[
(H_1^{c, \mathrm{sat}}(\mathcal{N}))^*=(H_1^{c}(\mathcal{N}))^*=BMO_c(\N),
\]
where $BMO_c(\N)$ denotes the noncommutative martingale column bounded mean oscillation space (see e.g. \cite{PX1997}).  We conclude the above argument as follows. 
\begin{theorem}\label{atomic}
	We have
	\[
	H_1^c(\N)= H_1^{c, \mathrm{sat}}(\mathcal{N})
	\]
	with equivalent norms. The same holds for row spaces. 
\end{theorem}

\subsection{Noncommutative Sunouchi operator}\label{sec-Sun23}
 The following theorem is the main result in this subsection.

\begin{theorem}\label{SO-1}
	Let $(n_k)_{k\geq1}\subseteq \mathbb{N}$ be such that $M_{k-1}\leq n_k< M_k$ for each $k\geq1$, and let $T_k=\sigma_{n_k}-\mathbb{E}_k$. For each $1\leq p\leq 2$, there exists a constant $c_{p}>0$ depending only on $p$ such that 
	\[
	\|(T_k(f))_{k\geq1}\|_{L_{p}(\N,\ell_2^{c})} \leq c_p\|f\|_{H_p^c(\N)}.
	\]
	The same result holds true for row spaces. 
\end{theorem}



\begin{remark} 
As mentioned in the introduction, there are rich results related to the Sunouchi operator in the commutative setting; see e.g. \cite{Su1951, Ga1993, Si2000, We2005}. 
  The noncommutative   Sunouchi operator deserves further research.
\end{remark}


To prove Theorem \ref{SO-1}, we need the following sequence of lemmas.
\begin{lemma}\label{T22}
	Let $(T_k)_{k\geq1}$ be operators as in Theorem \ref{SO-1}. Then, for every $f\in L_{2}(\mathcal{N})$, there exists a constant $c>0$ such that
	$$\|(T_kf)_{k\geq1}\|_{L_{2}(\mathcal{N},\ell_{2}^c)}\leq c\|f\|_{L_2(\mathcal{N})}.$$	
\end{lemma}
\begin{proof}
	According to the Parseval  identity,
	$$\|f\|_{L_2(\N)}^2=\sum_{j\geq0}\|\widehat{f}(j)\|_{L_2(\M)}^2, \quad \forall f\in L_2(\N),$$
	we have
	\begin{align*}
		\|(T_kf)_{k\geq1}\|_{L_{2}(\mathcal{N},\ell_{2}^c)}^2=\sum_{k\geq 1} \| T_k(f)\|_{L_{2}(\mathcal{N})}^2= \sum_{k\geq 1} \sum_{j\geq 0} \tau(|\widehat{T_k(f)}(j)|^2).
	\end{align*}
	Note that
	\begin{align*}
		\widehat{T_k(f)}(j)&= \widehat{K_{n_k}}(j) \widehat{f} (j)-\widehat{D_{M_k}}(j) \widehat{f} (j)= (\widehat{K_{n_k}}(j)-\widehat{D_{M_k}}(j)) \cdot \widehat{f} (j).
	\end{align*}
	Then
	$$\|(T_kf)_{k\geq1}\|_{L_{2}(\mathcal{N},\ell_{2}^c)}^2= \sum_{j\geq 0} \sum_{k\geq1} |\widehat{K_{n_k}}(j)-\widehat{D_{M_k}}(j)|^2\cdot \tau(|\widehat{f} (j)|^2). $$
	By the Parseval identity,  it suffices to show 
	\begin{equation}\label{j}
		\sum_{k\geq1} |\widehat{K_{n_k}}(j)-\widehat{D_{M_k}}(j)|^2\leq c,\quad \forall j\geq0.
	\end{equation}
	Observe that $\widehat{D_{M_k}}(j)=1$ if $j\leq M_k-1$ and $\widehat{D_{M_k}}(j)=0$ if $j\geq M_k$; $\widehat{K_{n_k}}(j)=(n_k-j)/n_k$ if $j\leq n_k$ and $\widehat{K_{n_k}}(j)=0$ if $j\geq n_k$. Then \eqref{j} is clear for $j=0$.  On the other hand,  for each fixed $j\geq1$, there exists an unique integer $l\in \mathbb{N}$ such that 
	$$M_l\leq j<M_{l+1}.$$
	Thus, 
	\begin{align*}
		\sum_{k\geq1} |\widehat{K_{n_k}}(j)-\widehat{D_{M_k}}(j)|^2&=\sum_{k\geq l+1} |\widehat{K_{n_k}}(j)-\widehat{D_{M_k}}(j)|^2\\
		&\leq \sum_{k\geq l+1}|\frac{j}{n_k}|^2 \\
&\leq \sum_{k\geq l+1}\big(\frac{M_{l+1}}{M_{k-1}}\big)^2\leq c.
	\end{align*}
	The proof is complete.
\end{proof}

\begin{lemma}\label{kn0}
	Let $(T_k)_{k\geq1}$ be as in Theorem \ref{SO-1},  and suppose that $a$ is  a $(1,2)_c$-simple  atom associated with some $n_0\geq1$. Then $T_k(a)=0$ for each $k\leq  n_0$.
\end{lemma}

\begin{proof}
	For $k\leq n_0$,  the fact that $a$ is a $(1,2)_c$-simple  atom  with respect to $n_0$ implies that $\mathbb{E}_k(a)=0$. Recall that $n_k<M_{k}$ for each $k$. Hence, for $k\leq n_0 $ and any $l\leq n_k$,  we have
	\[
	S_l(a)=S_l(S_{M_{n_0}}(a))= S_l(\mathbb{E}_{n_0}(a))=0,
	\]
	which implies $\sigma_{n_{k}}(a)=0$.
	By the definition of $T_k$, we have $T_k(a)=0$ provided  $k\leq n_0$.
\end{proof}

The following is the key lemma   in establishing Theorem \ref{SO-1}.
\begin{lemma}\label{nk-atom}
	Let $a$ be   a $(1,2)_c$-simple  atom associated with some $n_0\geq1$ and projection $e=e_Q\chi_Q$ with $Q\in D(\mathcal{F}_{n_0})$. If $k>n_0$ and $M_{k-1}\leq n_k<M_k$, then
	\[
\|\sigma_{n_k}(a)\chi_{Q^c}\|_{L_1(\N)}\leq c\Big[\delta^{n_0-k}+(k-n_0)\delta^{\frac{n_0-k}{2}}\Big],
\]
where $\delta$ is from Assumption \ref{gamma-nk} (iv).
\end{lemma}

\begin{proof}
	Similar to the argument before Theorem \ref{sigma-tilde}, we have 
	\begin{align*}
		\|\sigma_{n_k}(a)\chi_{Q^c}\|_{L_1(\N)}\leq\|\widetilde{\sigma}_{k-1}(|a|)\chi_{Q^c}\|_{L_1(\N)},
	\end{align*}
	where $\widetilde{\sigma}_{k-1}$ is defined as \eqref{s-tilde}.
	Hence
	\begin{align*}
		\|\widetilde{\sigma}_{k-1}(|a|)\chi_{Q^c}\|_{L_1(\N)}&= \int_{G_m\setminus Q}\|\int_Q|a|(s)\widetilde{K}_{k-1}(t,s)d\mu(s)\|_{L_1(\mathcal{M})}d\mu(t)\\
		&=\int_{G_m\setminus Q}\|A^{1/2} u(t)\big(\int_Q\widetilde{K}_{k-1}^2(t,s)d\mu(s)\big)^{\frac{1}{2}}\|_{L_1(\mathcal{M})}d\mu(t)\\
		&=\int_{G_m\setminus Q}\|A^{1/2} u(t)\|_{L_1(\mathcal{M})}\big(\int_Q\widetilde{K}_{k-1}^2(t,s)d\mu(s)\big)^{\frac{1}{2}}d\mu(t),
	\end{align*}
	where we used Lemma \ref{AcB} with $\sup_t \|u(t)\|_{L_{\infty}(\mathcal{M})}\leq 1$, 
	$$A=\int_Q |a(s)|^2ds.$$
	Since $a$ is a $(1,2)_c$-simple  atom and  $Q\in D(\mathcal{F}_{n_0})$, it follows from Remark \ref{use-rem} (ii) that
	\begin{align*} 
		\|\widetilde{\sigma}_{k-1}(|a|)\chi_{Q^c}\|_{L_1(\N)}&\leq  |Q|^{-1/2}\int_{G_m\setminus Q}\big(\int_Q\widetilde{K}_{k-1}^2(t,s)d\mu(s)\big)^{\frac{1}{2}}d\mu(t)\\
		&= |Q|^{-1/2}\sum_{j=0}^{n_0-1}  \int_{I_j(s)\setminus I_{j+1}(s)} \big(\int_Q\widetilde{K}_{k-1}^2(t,s)d\mu(s)\big)^{\frac{1}{2}}d\mu(t)\\
		&\leq |Q|^{-1/2}\sum_{j=0}^{n_0-1}  |I_j(s)\setminus I_{j+1}(s)|^{1/2} \Big(\int_{I_j(s)\setminus I_{j+1}(s)} \int_Q\widetilde{K}_{k-1}^2(t,s)d\mu(s)d\mu(t)\Big)^{\frac{1}{2}}. 
	\end{align*}
	Then by the Fubini theorem,  $k-1\geq n_0>j$ and Lemma \ref{K^2-E-a}, we get
	\begin{align*} 
		\|\widetilde{\sigma}_{k-1}(|a|)\chi_{Q^c}\|_{L_1(\N)}&\leq  |Q|^{-1/2}\sum_{j=0}^{n_0-1} M_j^{-\frac{1}{2}} \Big( \int_Q\int_{I_j(s)\setminus I_{j+1}(s)}\widetilde{K}_{k-1}^2(t,s)d\mu(t)d\mu(s)\Big)^{\frac{1}{2}}\\
		&\leq c |Q|^{-1/2} \sum_{j=0}^{n_0-1} M_j^{-\frac{1}{2}} |Q|^{\frac{1}{2}} \big(M_j^{\frac{1}{2}}\delta^{j-(k-1)}+(k-1-j)\delta^{\frac{j-(k-1)}{2}}M_j^{\frac{1}{2}}\big)\\
		&\leq c \sum_{j=0}^{n_0-1}\big(\delta^{j-(k-1)}+(k-1-j)\delta^{\frac{j-(k-1)}{2}}\big)\\
		&\leq c\Big[\delta^{n_0-k}+(k-n_0)\delta^{\frac{n_0-k}{2}}\Big],
	\end{align*}
	which completes the proof of the lemma.
\end{proof}

\begin{lemma}\label{lem-atom}
	Let $(T_k)_{k\geq1}$ be as in Theorem \ref{SO-1}, and	let $a$ be a $(1,2)_c$-simple  atom. Then 
	$$\|(T_ka)_{k\geq1}\|_{L_{1}(\N,\ell_{2}^c)}\leq c.$$
	The same result holds for $\|\cdot\|_{L_{1}(\N,\ell_{2}^r)}$ and row atoms. 
\end{lemma}
\begin{proof}
	Without loss of generality, we may assume that $a$ is a $(1,2)_c$-simple  atom associated with $n_0\geq1$ and projection $e=e_Q\chi_Q$ with $e_Q\in \mathcal{P}(\M)$, $Q\in D(\mathcal{F}_{n_0})$. 	Note that $\|\cdot\|_{L_{1}(\N,\ell_{2}^c)}$ is  a norm. 
	Then 
	\begin{align*}
		\|(T_ka)_{k\geq1}\|_{L_{1}(\N,\ell_{2}^c)}&=\Big\|\sum_{k\geq 1} T_k(a) \otimes e_{k,1}\Big\|_{L_1(\N\bar{\otimes}B(\ell_{2}))}\\
		&\leq \Big\|\sum_{k\geq 1} T_k(a) e \otimes e_{k,1}\Big\|_{L_1(\N\bar{\otimes}B(\ell_{2}))}+\Big\|\sum_{k\geq 1} T_k(a)(1-e) \otimes e_{k,1}\Big\|_{L_1(\N\bar{\otimes}B(\ell_{2}))}\\
		&\coloneqq \mathrm{I}+\mathrm{II},
	\end{align*}
	where $e_{k,1}$ stands for the column matrix taking value $1$ at the $(k,1)$-position and all others are $0$. The term $\mathrm{I}$ is  estimated as follows: 
	\begin{align*}
		\mathrm{I}&=\Big\| \sum_{k\geq1} e|T_k(a)|^2e\Big\|_{L_{\frac{1}{2}}(\N)}^{\frac{1}{2}} \leq c\|e\|_{L_1(\N)}^{1/2} \|a\|_{L_2(\N)}\leq c,
	\end{align*}
	where we used  Lemma \ref{T22}.

	Observe that,  if $k>n_0$, then $\mathbb{E}_k(a)(\1-e)=\mathbb{E}_k(ae)(\1-e)=\mathbb{E}_k(a) e(\1-e)=0$. Hence, by Lemma \ref{kn0}, we get
	$$\sum_{k\geq1}T_k(a)(\1-e)=\sum_{k> n_0} \sigma_{n_k}(a)(\1-e). $$
	On the other hand,  it is easy see that, for each $t\in Q$, 
	$$\sigma_{n_k}(a)(t)(\1-e)(t)=\sigma_{n_k}(a)(t)e_Q (\1_{\M}-e_Q)= 0, \quad k>n_0.$$
	Consequently, using Lemma \ref{nk-atom}, we arrive at
	\begin{align*}
		\mathrm{II}&\leq \sum_{k> n_0}  \|\sigma_{n_k}(a)\chi_{Q^c}\|_{L_1(\N)}\\
		&\leq c\sum_{k> n_0} c\Big[\delta^{n_0-k}+(k-n_0)\delta^{\frac{n_0-k}{2}}\Big] \leq c.
	\end{align*}
	The desired result follows from the estimates of $\mathrm{I}$ and $\mathrm{II}$. 
\end{proof}

Theorem \ref{SO-1} now follows immediately from the atomic decomposition and the complex interpolation between $H^{c}_{1}(\N)$ and $H^{c}_{2}(\N)$ ($=L_2(\N)$).

\begin{proof}[Proof of Theorem \ref{SO-1}]
	The case $p=1$  follows from Theorem \ref{atomic} and Lemma \ref{lem-atom} directly.   For $1<p<2$, combining the complex interpolation between $H_1^c(\N)$ and $L_2(\N)$ (see e.g. \cite[Lemma 4.3]{BCPY2010}), the case $p=1$ and Lemma \ref{T22}, we have
	\begin{equation}\label{ppc}
		\|(T_k(f))_{k\geq1}\|_{L_{p}(\N,\ell_2^{c})} \leq c_p\|f\|_{H_p^c(\N)},\quad 1<p<2.
	\end{equation}
	The same result holds true for row spaces, and hence the proof is complete.
\end{proof}

\subsection{Proof of Theorem \ref{main-asy-op}}\label{sec-pf asy op}

We now arrive at the position to prove Theorem \ref{main-asy-op}. 

\begin{proof}[Proof of Theorem \ref{main-asy-op}]
	Since every lacunary sequence $(n_{k})_{k\geq 1}$ could be split into a finite number of sub-sequences $\{(n^{j}_{k})_{k\geq 1}\}_{j=1}^{N}$ such that, $m_{k-2}\leq n^{j}_{k}/n^{j}_{k-1}$ for every $k\in \Na$ and $j=1,\dots,N$. Hence, without loss of generality, we assume that $(n_{k})_{k\geq 1}$ is a sequence satisfying $M_{k-1}\leq n_{k}<M_{k}$ for each $k\geq 1$.
	
	(i) 
	Note that $\sigma_{n_{k}}(f)=T_{k}(f)+\mathbb{E}_{k}(f)$ and $\|\cdot\|_{\Lambda_{p,\infty}(\N,\ell^{c}_{\infty})}$ is a quasi-norm.  Then, combining Lemma \ref{ebm-pc} and \cite[Theorem A]{HJP2016}, we obtain 
	\begin{align*}
		\left\|(\sigma_{n_{k}}(f))_{k\geq 1}\right\|_{\Lambda_{p,\infty}(\N,\ell^{c}_{\infty})}
		&\lesssim \left\|\left(T_{k}(f)\right)_{k\geq 1}\right\|_{\Lambda_{p,\infty}(\N,\ell^{c}_{\infty})}+\|(\mathbb{E}_{k}(f))_{k\geq1}\|_{\Lambda_{p,\infty}(\N,\ell^{c}_{\infty})}\\
		&\lesssim \left\|(T_{k}(f))_{k\geq 1}\right\|_{L_{p}(\N,\ell^{c}_{2})}+\|f\|_{H_p^c(\N)}\\
		&\leq  c_{p}\|f\|_{H_p^c(\N)},
	\end{align*}
	where the last inequality follows from   Theorem \ref{SO-1}.
	Analogous argument yields that the inequality still holds for row spaces, and hence we complete the proof of item (i).

	(ii) Let $f\in L_p(\mathcal{N})$ with $1<p<2$, and assume that $f=f^{c}+f^{r}$ with $f^c\in H_p^c(\mathcal{N})$ and $f^r\in H_p^r(\mathcal{N})$.  Combining Lemma \ref{ebm-pc} with Theorem \ref{SO-1}, we get 
	\begin{equation*}\label{decomposition 2}
		\begin{split}
			\left\|(\sigma_{n_{k}}(f^{c}))_{k\geq 1}\right\|_{L_{p}(\N,\ell^{c}_{\infty})}
			&\lesssim_p \left\|\left(T_{k}(f^{c})\right)_{k\geq 1}\right\|_{L_{p}(\N,\ell^{c}_{\infty})}+\|(\mathbb{E}_{k}(f^{c}))_{k\geq 1}\|_{L_{p}(\N,\ell^{c}_{\infty})}\\
			&\leq \left\|\left(T_{k}(f^{c})\right)_{k\geq 1}\right\|_{L_{p}(\N,\ell^{c}_{2})}+\|(\mathbb{E}_{k}(f^{c}))_{k\geq 1}\|_{L_{p}(\N,\ell^{c}_{\infty})}\\
			&\lesssim_p \|f^c\|_{H_p^c(\N)}+\|(\mathbb{E}_{k}(f^{c}))_{k\geq 1}\|_{L_{p}(\N,\ell^{c}_{\infty})}.
		\end{split}
	\end{equation*}
	Analogously, the same inequality holds for $f^{r}$, that is,
	\begin{equation*}
		\left\|(\sigma_{n_{k}}(f^{r}))_{k\geq 1}\right\|_{L_{p}(\N,\ell^{r}_{\infty})}\lesssim_p \|f^{r}\|_{H^{r}_{p}(\N)}+\|(\mathbb{E}_{k}(f^{r}))_{k\geq 1}\|_{L_{p}(\N,\ell^{r}_{\infty})}.
	\end{equation*}
	Thus, using the noncommutative Burkholder-Gundy inequality (see \cite{PX1997} or \eqref{decomposition 1}) and \cite[Theorem B]{HJP2016}, we have
	\begin{equation*}
		\inf\limits_{f=f^{c}+f^{r}}\left\{\|(\sigma_{n_{k}}(f^{c}))_{k\geq 1}\|_{L_{p}(\N,\ell^{c}_{\infty})}+\|(\sigma_{n_{k}}(f^{r}))_{k\geq 1}\|_{L_{p}(\N,\ell^{r}_{\infty})}\right\}\leq c_{p}\|f\|_{L_{p}(\N)}.
	\end{equation*}
	Our proof is complete. 
\end{proof}

\begin{remark} According to  \cite[Theorem 5.11]{CRX2023},    the following inequality holds for  regular filtration:
	$$\|\big(\mathbb{E}_n(f)\big)_{n\geq 1}\|_{L_{1}(\mathcal{N},\ell_{\infty}^c)}\leq c\|f\|_{H_1^c(\mathcal{N})}.$$
	From this inequality, we can get  a  strong type version of (i) in Theorem \ref{main-asy-op}  for $p=1$ as follows:
	$$\|(\sigma_{n_{k}}(f))_{k\geq1}\|_{L_{1}(\mathcal{N},\ell_{\infty}^c)}\leq c\|f\|_{H_1^c(\mathcal{N})}.$$
\end{remark}

The following remark states that Theorem \ref{main-asy-op} (i) holds in full generality whenever $p=2$. 

\begin{remark}\label{full-rng}
	We have 
	\[
	\left\|(\sigma_n(f))_{n\geq1}\right\|_{\Lambda_{2,\infty}(\mathcal{N},\ell_{\infty}^c)}\leq c \|f\|_{L_2(\mathcal{N})}.
	\]
	According to  Lemma \ref{AcB}, for any $f\in L_1(\mathcal{N})$, we have
	$$\sigma_n(f)(\eta)=\int_{G_m} K_n(\eta, t)f(t)dt= B^{1/2}uA^{1/2},$$
	where 
	$$B=\int_{G_m}|K_n(\eta, t)|dt,\quad A= \int_{G_m} |K_n(\eta, t)|\cdot|f(t)|^2dt, \quad \|u\|_{L_{\infty}(\mathcal{M})}\leq 1.$$
	Hence, by Lemma \ref{KI-E-lemma} (i), 
	\begin{align*}
		|\sigma_n(f)|^2=A^{1/2} u^*B^{1/2}B^{1/2}uA^{1/2}\leq c A\leq c\sigma_{n}^+(|f|^2),
	\end{align*}
where $\sigma_{n}^+(f)=\int_{G_m} |K_n(\eta, t)|f(t)dt$.	Consequently, 
	\begin{align*}
		\|(\sigma_n(f))_{n\geq1}\|_{\Lambda_{2,\infty}(\mathcal{N},\ell_{\infty}^c)}&=\|(|\sigma_n(f)|^2)_{n\geq1}\|_{\Lambda_{1,\infty}(\mathcal{N},\ell_{\infty})}^{1/2}\\
		&\leq \sqrt{2} \|(\sigma_n^+(|f|^2))_{n\geq1}\|_{\Lambda_{1,\infty}(\mathcal{N},\ell_{\infty})}^{1/2} \leq c  \||f|^2\|_{L_1(\mathcal{N})}^{1/2}=c\|f\|_{L_2(\mathcal{N})},
	\end{align*}
	where the last inequality is due to 
	$$\|(\sigma_n^+(|f|^2))_{n\geq1}\|_{\Lambda_{1,\infty}(\mathcal{N},\ell_{\infty})}^{1/2}\leq\|(\widetilde{\sigma}_l(|f|^2))_{l\geq1}\|_{\Lambda_{1,\infty}(\mathcal{N},\ell_{\infty})}^{1/2},\quad M_l\leq n<M_{l+1}$$
	and  Theorem \ref{sigma-tilde} (i).
\end{remark}


\subsection{Proof of Theorem \ref{main-asy}}
This subsection is devoted to proving  Theorem \ref{main-asy}. To this end,   we need to work with $G_{2m}$ for a given bounded sequence $m=\{m_k\}_{k\in \mathbb{N}}$, where $2m$ is defined as in \eqref{2m}. In this case, let $M_0=1$, $M_{1}=(2m)_0M_0$ and $M_{k}=(2m)_{k-1}M_{k-1}$ for each $k\geq2$.
Assume that $(n_{k})_{k\in \Na}\subseteq\Na$ satisfies $M_{2(k-1)}\leq n_{k}<M_{2k}$ for each $k$, and the operator $\Tr_{k}$ is given by
\begin{equation*}
	\Tr_{k}(x)\coloneqq\sigma_{n_{k}}^{\Ra}(x)-\Ex_{k}(x),\quad  x\in L_1(\Ra).
\end{equation*}

\begin{theorem}\label{sunouchi operator 2}
	For each $x\in H^{c}_{p}(\mathcal{R})$ with $1\leq p\leq 2$, there exists $c_{p}>0$ such that
	\begin{equation*}
		\left\|\left(\Tr_{k}(x)\right)_{k\geq 1}\right\|_{L_{p}(\mathcal{R},\ell^{c}_{2})}\leq c_{p}\|x\|_{H^{c}_{p}(\mathcal{R})}.
	\end{equation*}
	The same result holds for row spaces.
\end{theorem}

\begin{proof}
	By the density of $\mathrm{Poly}(\Ra)$ in $H^{c}_{p}(\Ra)$,  without loss of generality, we assume that $x\in \mathrm{Poly}(\Ra)$. Let $\gamma:\Ra\to \mathcal{A}\coloneqq L_{\infty}(G_{2m})\bar{\otimes}\Ra$ be the mapping given as in \eqref{transference mapping}.   For each $k\in \Na$, define $T_{2k}$ as follows
	\begin{equation*}
		T_{2k}(\gamma(x))\coloneqq \sigma_{n_k}(\gamma(x))-\mathbb{E}_{2k}(\gamma(x)),
	\end{equation*}
	where the conditional expectation $\mathbb{E}_{2k}$ is referred to Example \ref{one-d} (with respect to $G_{2m}$). 
Applying Theorem \ref{SO-1} (at least its proof works well), we get 
	\begin{equation}\label{transference 6}
		\left\|\left(T_{2k}(\gamma(x))\right)_{k\geq 1}\right\|_{L_{p}(\mathcal{A},\ell^{c}_{2})}\leq c_p\|\gamma(x)\|_{\widetilde{H}^{c}_{p}(\mathcal{A})},
	\end{equation}
where  $\widetilde{H}^{c}_{p}(\mathcal{A})$ denotes the noncommutative martingale Hardy space    with respect to the filtration $(L_{\infty}(G_{2m},\mathcal{F}_{2k})\bar{\otimes}\Ra)_{k\geq1}.$ Combining \eqref{EkSk}, \eqref{S2n}, and the definition of $\gamma$, we have
	\begin{equation}\label{gammaEk}
		\gamma(\Ex_{k}(x))= S_{M_{2k}}(\gamma(x))=\mathbb{E}_{2k}(\gamma(x)), 
	\end{equation}
which further implies that   
	\begin{equation*}
		\gamma(dx_{k})=\mathbb{E}_{2k}(\gamma(2x))-\mathbb{E}_{2(k-1)}(\gamma(x)).
	\end{equation*}
Since $\gamma$ is a $\ast$-homomorphism, it follows from Theorem \ref{transference argument} that
	\begin{equation}\label{transference 8}
\begin{aligned}
	\|\gamma(x)\|_{\widetilde{H}^{c}_{p}(\mathcal{A})}
	&=\left\|\left(\sum\limits_{k\geq 1}|\mathbb{E}_{2k}(\gamma(x))-\mathbb{E}_{2(k-1)}(\gamma(x))|^{2}\right)^{1/2}\right\|_{L_{p}(\mathcal{A})}\\
&=\left\|\gamma\left(\sum\limits_{k\geq 1}|dx_{k}|^{2}\right)\right\|^{1/2}_{L_{p/2}(\mathcal{A})}\\
&=\left\|\left(\sum\limits_{k\geq 1}|dx_{k}|^{2}\right)^{1/2}\right\|_{L_{p}(\mathcal{R})}=\|x\|_{H^{c}_{p}(\Ra)}.
\end{aligned}
	\end{equation}
On the other hand, by the definition of $\gamma$ and \eqref{gammaEk}, it is easy to verify that
	\begin{equation*}
		\gamma(\Tr_{k}(x))=T_{2k}(\gamma(x)).
	\end{equation*}
	Similarly to the argument in \eqref{transference 8}, we have
	\begin{equation}\label{transference 9}
		\left\|\left(T_{2k}(\gamma(x))\right)_{k\geq 1}\right\|_{L_{p}(\mathcal{A},\ell^{c}_{2})}=
		\left\|\left(\Tr_{k}(x)\right)_{k\geq 1}\right\|_{L_{p}(\Ra,\ell_{2}^c)}.
	\end{equation}
	Substituting \eqref{transference 8} and \eqref{transference 9} into \eqref{transference 6} yields the desired inequality. Analogously, same arguments are valid for row spaces.
\end{proof}

\begin{corollary}
	For each $x\in L_p(\mathcal{R})$ with $1<p\leq 2$, there is a decomposition $x=x^{c}+x^{r}$ such that
	\begin{equation*}
		\left\|(\Tr_{k}(x^{c}))_{k\geq 1}\right\|_{L_{p}(\Ra,\ell^{c}_{2})}+\left\|(\Tr_{k}(x^{r}))_{k\geq 1}\right\|_{L_{p}(\Ra,\ell^{r}_{2})}\leq c_{p}\|x\|_{L_{p}(\Ra)}.
	\end{equation*}
\end{corollary}

\begin{proof}
	For the case $1<p<2$, it has been established in \cite[Theorem 4.10]{JRWZ2020} that there exists a constant $c^{\prime}_{p}>0$ depending only on $p$ such that for each $x\in L_{p}(\Ra)$ there exists a decomposition $x=x^{c}+x^{r}$ satisfying
	\begin{equation}\label{decomposition 1}
		\|x^{c}\|_{H^{c}_{p}(\Ra)}+\|x^{r}\|_{H^{r}_{p}(\Ra)}\leq c^{\prime}_{p}\|x\|_{L_{p}(\Ra)}.
	\end{equation}
	Then the first desired inequality follows from Theorem \ref{sunouchi operator 2} directly. 
	The case $p=2$ is due to Lemma \ref{T22}. 
\end{proof}

Now we prove Theorem \ref{main-asy}. 

\begin{proof}[Proof of Theorem \ref{main-asy}]
Using Theorem \ref{sunouchi operator 2} instead of Theorem \ref{SO-1},   and repeating the argument used in the proof of Theorem \ref{main-asy-op} in Section \ref{sec-pf asy op}, we can finish the proof. 
\end{proof}

\begin{remark}
Note that to prove Theorem \ref{main-asy}, we first transfer Theorem \ref{SO-1} to its noncommutative corresponding result Theorem \ref{sunouchi operator 2}, and then repeat the proof of Theorem \ref{main-asy-op} to finish the proof. At the time of this writing, we are not clear that whether we can transfer Theorem \ref{main-asy-op} to Theorem \ref{main-asy} directly.
\end{remark}

We conclude this paper with two  questions.

\begin{question}\label{unbdd-V} It is an interesting problem to consider Theorem \ref{CS}  for unbounded Vilenkin system. We also point out here that for unbounded Vilenkin system even in the classical setting Theorem \ref{CS} (i) is a long time open problem.  The best related result was proved by   G\'{a}t  \cite{Ga2007}  for so-called rarely unbounded Vilenkin groups. 
\end{question}

\begin{question}\label{nc-2-adic} 
Recall that the construction of noncommutative Walsh and Vilenkin systems appear in the series of the papers by Sukochev and his collaborators \cite{AFS1996,CPS2013, DS2000,DFdePS2001,SF1995}. 
It is natural to ask whether 
we can construct a noncommutative $2$-adic integer system (the classical version is in Example \ref{m-adic-example})?  
\end{question}

\noindent{\bf Acknowledgement.}
Yong Jiao is partially supported by the NSFC (No. 12125109, No. 11961131003). Sijie Luo is partially supported by the NSFC (No. 12201646), Natural Science Foundation Hunan (No. 2023JJ40696). Dejian Zhou is   partially supported by NSFC (No. 12001541),  Natural Science Foundation Hunan (No. 2023JJ20058), CSU
Innovation-Driven Research Programme (No. 2023CXQD016).

\begin{appendices}
	
	\section{Almost uniform convergence}
	In this subsection, we recall definitions of almost uniform convergence in semifinite von Neumann algebras. The following definition, which goes back at least   Lance \cite{La1976} or \cite{Ja1985},  is taken from \cite{CGP2020} and \cite[p.\,109]{De2011}.
	\begin{definition}\label{npc}
		Consider $(x_k)_{k\geq1}\subseteq L_0(\mathcal{M})$ and $x\in L_0(\mathcal{M})$.
		\begin{enumerate}[{\rm (i)}]
			\item We say $(x_k)_{k\geq1}$ converges to $x$ column almost uniformly (c.a.u. in short) if for any $\varepsilon>0$, there is a projection $e\in \mathcal{P}(\mathcal{M})$ such that
			\[
			\tau({\bf 1}-e)<\varepsilon,\quad \lim_{k\to\infty}\|(x_k-x)e\|_{\mathcal{M}}=0.
			\]
			\item We say $(x_k)_{k\geq1}$ converges to $x$ row almost uniformly (r.a.u. in short)  if $(x_k^*)_{k\geq1}$ converges to $x^*$ column almost uniformly.
			
			\item We say $(x_k)_{k\geq1}$ converges to $x$ column $+$ row almost uniformly ($c+r$ a.u.) provided the sequence $(x_{k}-x)_{k\geq 1}$ decomposes into a sum $(a_{k})_{k\geq 1}+(b_{k})_{k\geq 1}$ of two sequences such that $(a_{k})_{k\geq 1}$ converges to $0$ column almost uniformly and $(b_{k})_{k\geq 1}$ converges to $0$ row almost uniformly.
			
			\item We say $(x_k)_{k\geq1}$ converges to $x$ bilaterally almost uniformly (b.a.u. in short) if for any $\varepsilon>0$, there is a projection $e\in \mathcal{P}(\mathcal{M})$ such that
			\[
			\tau({\bf 1}-e)<\varepsilon,\quad \lim_{k\to\infty}\|e(x_k-x)e\|_{\mathcal{M}}=0.
			\] 
		\end{enumerate}
	\end{definition}
	The column almost uniform convergence is the usual almost uniform convergence used in \cite{CXY2013}, \cite{JX2007},    \cite{Me2007} and more; hence we sometimes still use the same symbol a.u. for short.   It was shown in \cite[Chapter 3.1.8]{De2011} that  only the trivial implications hold:
	$$\mbox{a.u. or r.a.u.}  \Rightarrow c+r \mbox{ a.u.} \Rightarrow \mbox{b.a.u.}$$
	The above definition  generalizes the notion of almost everywhere convergence in the case of finite
	measure spaces. Actually, for finite abelian von Neumann algebra $\mathcal{M}$ (i.e., $\tau(\1)<\infty$), the almost uniform convergences in the definition above are all equivalent
	to the usual almost everywhere convergence by virtue of Egorov's theorem.

	
\section{Proofs of transference results }

\begin{proof}[Proof of Proposition \ref{transference 1}]
(i) Since the linearity and $*$-preserving of $\gamma$ are obvious, it suffices to show the following
\[
\gamma(W_mW_n)=\gamma(W_m)\gamma(W_n).
\]
This actually	follows from the definition of $\gamma$, the second equality of \eqref{w-plus}	and Lemma \ref{RW-plus}.

Item (ii) is clear, and it now remains to prove item (iii). The fact that $\gamma$ is an    injective mapping from $\mathrm{Poly}(\Ra)$ into $L_{\infty}(G_{2m})\bar{\otimes} \mathcal{R}$ is clear. Indeed, for every $x\in \mathrm{Poly}(\Ra)$ with $x\not=0$,  there exists $j\in \mathbb{N}\cup\{0\}$ such that $\widehat{x}(j)\not=0$. Since $\{\psi_{j}\otimes W_{j}\}_{j=0}^{\infty}$ is linearly independent, it follows that $\gamma(x)\not=0$, which verifies our claim. The linear independence of $\{\psi_{j}\otimes W_{j}\}_{j=0}^{\infty}$ could be evidenced from the fact that $\{\psi_{j}\otimes W_{k}\}_{j,k=0}^{\infty}$ forms an orthonormal basis of $L_{2}(L_{\infty}(G_{2m})\bar{\otimes} \mathcal{R})$.
	
We now verify that $\gamma$ is normal. By the linearity of $\gamma$, it suffices to verify $\gamma(x_{\alpha})\to 0$ in weak-$\ast$ topology whenever the net $x_{\alpha}\to 0$ in weak-$\ast$ topology. Thanks to Grothendieck's dual characterization of completeness, it suffices to verify that $\gamma(x_{\alpha})\to 0$ in weak-$\ast$ topology whenever $\{x_{\alpha}\}_{\alpha\in\Gamma}$ is weak-$\ast$-null norm bounded net
(see e.g. \cite[p.\,149]{T-V-S}). Without loss of generality, assume that $\|x_{\alpha}\|_{\mathcal{R}}\leq 1$ and $x_{\alpha}=\sum_{j\geq0}\widehat{x}_{\alpha}(j)W_{j_{\alpha}}$ for each $\alpha$. It suffices to show the following
\begin{equation}\label{simplification}
\langle \psi_{j}\otimes W_{k},\gamma(x_{\alpha}) \rangle_{L_2(\mathcal{R})}\to 0,\quad \alpha\in \Gamma,
\end{equation}
for each $j$, $k\in\mathbb{N}\cup\{0\}$. Since $\{\psi_{j}\otimes W_{k}\}_{j,k=0}^{\infty}$ is orthonormal in $L_{2}\left(L_{\infty}(G_{2m})\bar{\otimes} \mathcal{R}\right)$, we have the following
\begin{equation*}
\langle \psi_{j}\otimes W_{k},\gamma(x_{\alpha}) \rangle_{\varphi}=
\begin{cases}
\hat{x}_{\alpha}(j),~j=k\\
0,\qquad\mbox{otherwise}.
\end{cases}
\end{equation*}
Note that $\mathcal{R}$ admits a predual $L_{1}(\mathcal{R})$. Then, by the assumption that $\{x_{\alpha}\}_{\alpha\in \Gamma}$ is weak-$\ast$ null in $\mathrm{Poly}(\Ra)$, we get that, for each $j\in \mathbb{N}\cup\{0\}$,
$\widehat{x}_{\alpha}(j)\to 0.$
Indeed, for every $j\in \mathbb{N}\cup\{0\}$,
\begin{equation*}
\begin{split}
\widehat{x}_{\alpha}(j)&=\tau\left(x_{\alpha}\cdot W^{*}_{j}\right)\\
&=u_j\cdot\tau\left(x_{\alpha}\cdot W_{j}\right)\to 0,
\end{split}
\end{equation*}
where $|u_j|=1$. This establishes that $\gamma$ is normal. Since $\gamma$ is injective, repeating the statement as above we can see  that $\gamma^{-1}$ is also normal.  
\end{proof}

We shall simply denote $\widehat{\mathcal{R}}\coloneqq\gamma(\mathcal{R})$ in the sequel.

\begin{proof}[Proof of Theorem \ref{transference argument}]
	For $p=\infty$, by Proposition \ref{transference 1-2}, $\gamma:\mathcal{R}\to \widehat{\mathcal{R}}\subseteq L_{\infty}(G_{2m})\bar{\otimes} \mathcal{R}$ is a $*$-isomorphism such that $\widehat{\mathcal{R}}$ is weak-$*$ closed in $L_{\infty}(G_{2m})\bar{\otimes} \mathcal{R}$. Hence, $\gamma$ is a $*$-isomorphism from $\mathcal{R}$ onto the von Neumann algebra $\widehat{\mathcal{R}}$, and therefore $\gamma$ is an isometry, that is, $\|\gamma(x)\|_{L_{\infty}(G_{2m})\bar{\otimes} \mathcal{R}}=\|x\|_{\mathcal{R}}$, for all $x\in \mathcal{R}$.
	
Since $\gamma:\mathcal{R}\to \widehat{\mathcal{R}}$ is a $*$-isomorphism, for every polynomial $g$ and positive element $x$ in $\mathcal{R}$, we have $\gamma(g(x))=g(\gamma(x))$. For every bounded continuous function $f$ on $[0,\infty)$, we get from the Weierstrauss theorem that $f$ could be approximated by a sequence of polynomials $(p_{n})_{n=1}^{\infty}$ in the compact-open topology. Since the spectral of $x$ is compact, we get $p_{n}(x)\to f(x)$ and $\gamma(p_{n}(x))\to \gamma(f(x))$   in $\mathcal{R}$.  
	
On the other hand, for each $n\in \mathbb{N}$, $\gamma(p_{n}(x))=p_{n}(\gamma(x))$. Then, applying the fact that $p_{n}\to f$ in the compact-open topology and the compactness of the spectral of $\gamma(x)$, we obtain
	\begin{equation*}
		f(\gamma(x))=\lim\limits_{n\to \infty}p_{n}(\gamma(x))=\lim\limits_{n\to \infty}\gamma(p_{n}(x))=\gamma(f(x)),
	\end{equation*}
where limits are taken with respect to the norm in $\mathcal{R}$. Furthermore, notice that $\gamma:\mathcal{R}\to \widehat{\mathcal{R}}$ is a $*$-isomorphism,  for arbitrary $x\in \mathcal{R}$, we have
	\begin{equation}\label{homomorphism 1}
		\gamma(|x|^{2})=\gamma(x^{*}x)=\gamma(x^{*})\gamma(x)=\gamma(x)^{*}\gamma(x)=|\gamma(x)|^{2}.
	\end{equation}
Since the   the mapping $t\to t^{p/2}$ is continuous, by \eqref{homomorphism 1}, we get
	\begin{equation*}
		|\gamma(x)|^{p}=\gamma(|x|^{p}),\quad \forall x\in \mathcal{R}.
	\end{equation*}
Then, it follows from Proposition \ref{transference 1} (ii) that
	\begin{equation*}
		\|\gamma(x)\|^{p}_{L_{p}(L_{\infty}(G_{2m})\bar{\otimes}\mathcal{R})}=\|x\|^{p}_{L_{p}(\mathcal{R})}.
	\end{equation*}
	
To prove \eqref{mu-est}, it suffices to show for each $\lambda>0$ and $x\in \mathcal{R}$,
\begin{equation*}
\varphi\left(\chi_{(\lambda,\infty)}(\gamma(x))\right)=\tau\left(\chi_{(\lambda,\infty)}(x)\right).
\end{equation*}
	Note that for arbitrary $\lambda>0$, the bounded measurable function $\chi_{(\lambda,\infty)}$ could be approximated by a sequence of bounded continuous functions $(f_{n})_{n=1}^{\infty}$ pointwise. Hence, $f_{n}(x)\to\chi_{(\lambda,\infty)}(x)$ in weak-$\ast$ topology. Indeed, for every $u\in L_{1}(\mathcal{R})$, by the spectral theorem (or, functional calculus), we have
	\begin{equation*}
		\begin{aligned}
			\tau\left(u\cdot f_{n}(\gamma(x))\right)&=\int_{spec(\gamma(x))}f_{n}(t)~\langle u,\Pi_{\gamma(x)}(dt)\1\rangle_{L_2(\Ra)}\\
			&\to \int_{spec(\gamma(x))}\chi_{(\lambda,\infty)}(t)~\langle u,\Pi_{\gamma(x)}(dt) \rangle_{L_2(\Ra)}\\
			&=\tau(u\cdot \chi_{(\lambda,\infty)}(\gamma(x)))\quad (n\to \infty),
		\end{aligned}
	\end{equation*}
where $\Pi_{\gamma(x)}$ is the associated spectral measure of $\gamma(x)$.
Since $\gamma$ is normal,  repeating arguments as above, we have
	\begin{equation*}
		\tau\left(u\cdot \gamma(f_{n}(x))\right)\to\tau\left(u\cdot \gamma(\chi_{(\lambda,\infty)}(x))\right),~(n\to \infty).
	\end{equation*}
Note that $\gamma$ is  a homomorphism, and hence $\gamma(f_{n}(x))=f_{n}(\gamma(x))$ for each $n\in \mathbb{N}$. Since $L_{1}(\mathcal{R})$ separates  points of $\mathcal{R}$, we have
	\begin{equation}\label{transference 2}
		\gamma(\chi_{(\lambda,\infty)})(x)=\chi_{(\lambda,\infty)}(\gamma(x)).
	\end{equation}
	Applying the trace preseving of $\gamma$ to \eqref{transference 2}, we obtain the desired equality.
\end{proof}
\end{appendices}

%


\providecommand{\bysame}{\leavevmode\hbox to3em{\hrulefill}\thinspace}
\providecommand{\MR}{\relax\ifhmode\unskip\space\fi MR }
\providecommand{\MRhref}[2]{%
  \href{http://www.ams.org/mathscinet-getitem?mr=#1}{#2}
}
\providecommand{\href}[2]{#2}

\end{document}